\documentclass[reqno]{amsart}
\usepackage{amsthm,amssymb,amsmath,calc,eucal,ifthen}
\usepackage{amscd}
\usepackage[all,arc,curve,color,frame,matrix,arrow]{xy}

\numberwithin{equation}{section}

\newtheorem{theorem}{Theorem}[section]
\newtheorem{lemma}[theorem]{Lemma}
\newtheorem{proposition}[theorem]{Proposition}

\theoremstyle{remark}
\newtheorem{remark}[theorem]{Remark}
\newtheorem{example}[theorem]{Example}

\newtheoremstyle{rmdefinition}{}{}{\upshape}{}{\bfseries}{.}{ }{}

\theoremstyle{rmdefinition}
\newtheorem{definition}[theorem]{Definition}

\newcommand{\C}{{\mathbb C}}

\newcommand{\be}[1]{\begin{equation}\label{#1}}
\newcommand{\ee}{\end{equation}}
\newcommand{\beqa}{\begin{eqnarray}}
\newcommand{\eeqa}{\end{eqnarray}}

\newcommand{\smthr}[3]{\mathop{\sum}\limits_{
	\mathop{}\limits^{\text{\scriptsize $#1$}}_{
		\mathop{}\limits^{\text{\scriptsize $#2$}}_{
		  \text{\scriptsize $#3$}}}}}
\newcommand{\smtwo}[2]{\mathop{\sum}\limits_{
	\mathop{}\limits^{\text{\scriptsize $#1$}}_{
	  \text{\scriptsize $#2$}}}}

\def\blt{\text{\scriptsize ${\hspace{-1.1pt}}\bullet{\hspace{-1.1pt}}$}}
\def\Blt{\text{\scriptsize $\bullet$}}
\def\TRE{\mathcal{T}}
\def\CET{\text{\rm C.E.T.}}
\def\CCW{\mathfrak{C}}
\def\CCw{\mathfrak{G}}

\def\opr{o}
\newcommand{\OPR}[1]{{#1}^{\opr}}

\newcounter{tmpc}
\newlength{\tmplenght}
\newlength{\tmplenghta}
\newlength{\tmplenghtb}
\newlength{\tmplenghtc}

\newenvironment{LIST}[1]{%
\setlength{\tmplenghta}{#1}
\setlength{\tmplenghtb}{#1}
\setlength{\tmplenghtc}{#1}
\advance\tmplenghtb-5pt
\advance\tmplenghtc 42pt
\setcounter{tmpc}{0}
\begin{list}{{\rm (\alph{tmpc})}}{\usecounter{tmpc}
\setlength{\leftmargin}{\tmplenghta}
\setlength{\rightmargin}{0cm}
\setlength{\itemsep}{1pt}
\setlength{\topsep}{3pt}
\setlength{\labelsep}{5pt}
\setlength{\labelwidth}{\tmplenghtb}
\setlength{\listparindent}{\tmplenghta}}
}{\end{list}}

\newcommand{\mbf}[1]{\ensuremath{\mathchoice
                    {\mbox{\boldmath$\displaystyle\mathbf{\mathit{#1}}$}}
                    {\mbox{\boldmath$\textstyle\mathbf{\mathit{#1}}$}}
                    {\mbox{\boldmath$\scriptstyle\mathbf{\mathit{#1}}$}}
                    {\mbox{\boldmath$\scriptscriptstyle\mathbf{\mathit{#1}}$}}}}

\begin{document}

\title[Curved $A_{\infty}$-algebras and Chern classes]{Curved $A_{\infty}$-algebras and Chern classes}

\author[NIKOLAY M. NiKOLOV\hspace{6pt}]{NIKOLAY M. NiKOLOV${}^{\, 1}$ \hspace{0pt}}
\author[\hspace{7pt}SVETOSLAV ZAHARIEV]{\hspace{7pt}SVETOSLAV ZAHARIEV${}^{\, 2}$}

\maketitle

\begin{center}\small
\parbox{286pt}{
\begin{LIST}{7pt}
\item[${}^1$]
${}$\hspace{-5pt}
INRNE, Bulgarian Academy of Sciences, \\ Tsarigradsko chussee 72 Blvd., Sofia 1784, Bulgaria;\\
mitov$@$inrne.bas.bg
\item[${}^2$]
${}$\hspace{-5pt}
LaGuardia Community College of The City University of New York,\\
MEC Department, 31-10 Thomson Ave.\\
Long Island City, NY 11101, U.S.A.;\\szahariev@lagcc.cuny.edu
\end{LIST}}
\end{center}

\bigskip

\begin{abstract}
We describe two constructions giving rise to curved $A_{\infty}$-algebras. The first consists of deforming $A_{\infty}$-algebras, while the second involves transferring curved dg structures that are deformations of (ordinary) dg structures along chain contractions.
As an application of the second construction, given a vector bundle on a polyhedron $X$, we exhibit a curved $A_{\infty}$-structure on
the complex of matrix-valued cochains of sufficiently fine triangulations of $X$.
We use this structure as a motivation to develop a homotopy associative version of Chern-Weil theory.
\end{abstract}



\section{Introduction}

Developing discrete analogues of all fundamental differential geometric structures is desirable for several reasons. Such discretization techniques provide efficient tools for explicit numerical computations needed in applications. They also lead to simple and transparent proofs of combinatorial analogues of many classical results in differential geometry (see e.g. R. Forman's work on discrete Morse theory \cite{F}) or to new results obtained by passing to the ``continuum limit'', as is the case with constructive quantum  gauge field theory based on lattice approximations (see the survey \cite{J} and the references therein.)

In this paper, we propose notions of connection and curvature on a piece-wise smooth complex vector bundle over a polyhedron based on the notion of a curved $A_{\infty}$-algebra and derive explicit combinatorial formulas for the real Chern classes of the bundle in terms of the curvature. Our starting point is the recent observation of D. Sullivan (\cite[Appendix]{TZ}, see also \cite{CG}) and \cite{DMS}) that the complex of simplicial cochains on a polyhedron possesses a canonical $A_{\infty}$-structure.

We recall that $A_{\infty}$-algebras \cite{S} are a generalization of differential graded (dg) algebras in which the associativity condition is replaced by an infinite sequence of identities involving higher ``multiplications''. Curved $A_{\infty}$-algebras \cite{GJ} (or weak
$A_{\infty}$-algebras in the terminology of \cite{Ke}) are a natural generalization of both  $A_{\infty}$-algebras and curved dg algebras \cite{Po}. The latter are graded associative algebras endowed with a degree 1 derivation whose square equals the commutator with a fixed closed element of degree 2.

Let $E$ be a complex vector bundle over a finite polyhedron $X$. For each piece-wise smooth idempotent $P$ representing the $K$-theory class of $E$  we construct a canonical curved $A_{\infty}$-structure on the space of matrix-valued simplicial cochains on a sufficiently fine subdivision of $X$. This structure is obtained in Section 3 by transferring a curved dg algebra structure on the space of piece-wise smooth matrix-valued differential forms on $X$ using homological perturbation theory; its elements are given by convergent infinite series.

Various algebraic generalizations of the classical Chern-Weil theory have appeared during the last few decades. We mention, among others, the Connes-Karoubi-Chern character in cyclic homology (\cite{Ka},\cite{Co}) in the context of associative algebras and Huebschmann's work \cite{H4} in the context of Lie-Rinehart algebras. We develop a homotopy associative extension of the Chern-Weil construction for curved dg algebras from \cite{Po} in which the input data is a curved $A_{\infty}$-algebra equipped with certain additional structure. As an application, in Section 4 we establish our main result which can be stated as follows.

\smallskip

\noindent{\bf Theorem.} There exist explicit formulae for simplicial cochains representing the real Chern classes of the bundle $E$ in terms of the
curved $A_{\infty}$-structure on matrix-valued cochains discussed above.

\smallskip

We point out that the problem of discretizing the vector bundle itself is not treated in this paper. We conjecture that it is possible to construct a
curved $A_{\infty}$-algebra out of a purely combinatorial data representing the bundle and then obtain the characteristic classes from this algebra.
We also hope that the notions of connection and curvature presented in this work might be useful in developing an alternative approach to lattice gauge theory and to the problem of constructing continuous quantum gauge fields.

We note that formulas for simplicial representatives of the integral Chern classes of a principal bundle have been obtained in \cite{PS}  by completely different methods, using a particular model of the classifying space of the general linear group.

\smallskip

Let us describe the contents of the paper in more detail.

In Section 2, after recalling some basic definitions we describe a general procedure of deforming $A_{\infty}$-algebras to curved ones using infinite series along a ``connection'', i.e. an arbitrary element of degree 1. We use the opportunity to derive analogous results for $L_{\infty}$-algebras, the homotopy associative generalization of Lie dg algebras, building on an observation from \cite{G}.

In Section 3, we begin with a brief review of homological perturbation theory and then show how to use it to transfer curved dg structures that are deformations of ordinary dg algebras in the above sense along contractions. We also generalize the sum over rooted planar tree formulas of Kontsevich and Soibelman
\cite{KS} to this setting. As an application, we obtain the curved $A_{\infty}$-structure on matrix-valued simplicial cochains discussed above.

In Section 4, we develop an algebraic homotopy associative Chern-Weil theory. We introduce the notion of a Chern-Weil triple, a curved $A_{\infty}$-algebra together with a map into a chain complex satisfying certain symmetry properties. We define the Chern character of a Chern-Weil triple, show that it is natural with respect to an appropriate class of morphisms and finally prove the theorem stated above.

\section{Curved $A_{\infty}$- and $L_{\infty}$-structures. Deformations}

\subsection{Basic definitions}

Throughout this paper, we will work with non-negatively graded vector spaces over a fixed field $\textbf{k}$
that is either the real or the complex numbers.
By a chain complex we always mean a graded vector space equipped with a differential of degree 1.
We write $V^{\otimes{k}}$ for the $k$-th ({\it graded})
tensor power of a graded vector space $V=\bigoplus_{p}V_{p}$ so that $V^{\otimes{0}}=\textbf{k}$.
When applying graded maps to graded objects, we always use the {\it Koszul sign rule},
by which we mean the appearance of a sign $ (-1)^{\text{deg}(a)\,\text{deg}(b)}$
when switching two adjacent graded symbols $a$ and $b$.

\begin{definition}\label{curvedAinfinity}
A {\bf curved $A_{\infty}$-structure}
on a graded vector space $V$
is given by a family of linear maps
$$
{m_{k}}:
V^{\otimes{k}}\rightarrow V, \quad k \, = \, 0,1,2,\ldots
$$
satisfying for every $n\geqslant 0$ the identity
\be{Ainfide}
\sum_{r+s+t \, = \, n}(-1)^{r+st} \, m_{r+t+1}
\left(\textbf{1}_{V}^{\otimes{r}}\otimes m_{s} \otimes \textbf{1}_{V}^{\otimes{t}}\right) \, = \, 0 \,.
\ee
The map $m_{k}$ is assumed to have degree $2-k$.
\end{definition}

It appears that the term curved
$A_{\infty}$-algebra $V$ has been used for first time in \cite{N}. In the special case when the map $m_{0}:\textbf{k} \rightarrow V_{2}$ is $0$ we obtain
the usual notion of $A_{\infty}$-algebra (see \cite{S},\cite{GJ} and the surveys \cite{H3} and \cite{Ke}).
We will (by a slight abuse of notation) denote the degree 2 element of $V$ that is the image of
$1\in \textbf{k}$ under  $m _{0}$ also by $m_{0}$ and  we will refer to it as the {\bf curvature} of the curved
$A_{\infty}$-algebra $V$.
This terminology is motivated by the following example. When  $m_{k}=0$  for all $k>2$ we obtain the definition of {\bf curved differential graded (dg) algebra}
\cite{Po}.
In this case, the first four identities in Eq. (\ref{Ainfide}) read
\be{cdg1}
m_{1} m_{0} \, = \, 0 \,,
\ee
\be{cdg2}
m_{1} m_{1} \, = \, m_{2} (m_{0} \otimes \textbf{1})-m_{2}(\textbf{1} \otimes m_{0}) \,,
\ee
\be{cdg3}
m_{1} m_{2} \, = \, m_{2}(m_{1} \otimes \textbf{1})-m_{2}(\textbf{1} \otimes m_{1}) \,,
\ee
\be{cd4}m_{2} (m_{2} \otimes \textbf{1}) \, = \, m_{2}(\textbf{1} \otimes m_{2}) \,.
\ee
Eq. (\ref{cdg1}) can be interpreted as an abstract Bianchi identity,
Eq. (\ref{cdg2}) says that square of the ``covariant derivative'' $m_{1}$ equals to a commutator with
the curvature, (\ref{cdg3}) expresses the fact that $m_{1}$ is a (graded) derivation
and the last identity amounts to the associativity of the algebra.

In this paper, we will be mostly interested in examples of curved dg algebras arising from vector bundles.

\begin{example}\label{ex1section1}
Let $E$ be a real or complex vector bundle with connection $\nabla$ over a smooth manifold $M$ and
let $\text{End}(E)$ denote the bundle of endomorphisms of $E$.
The graded space $\Omega^{\bullet}(M,\text{End}(E))$
of differential forms on $M$ with values in
$\text{End}(E)$ becomes a curved dg algebra by taking $m_{0}$ to be the curvature of $\nabla$,
$m_{1}$ to be the induced connection on $\text{End}(E)$ and $m_{2}$ to be the wedge product
combined with the composition on $\text{End}(E)$.
\end{example}

\begin{example}\label{ex2section1}
Let $A$ be (an ordinary) dg algebra with differential $d$ and let $\gamma$ be a fixed degree 1 element in $A$.
It is easy to check that we can {\em deform}
$A$ to a curved dg algebra as follows.
Set
$$
m_{0} \, = \, d\gamma+\gamma . \gamma \,,
$$
$$
m_{1} \, = \, d+\bigl[\gamma,\cdot\bigr]
$$
and take $m_{2}$ to be the multiplication ``$\cdot$'' in $A$
(here $\bigl[\cdot,\cdot\bigr]$ denotes the (graded) commutator with respect to the multiplication in $A$).
\end{example}

Clearly when the bundle $\text{End}(E)$ is trivial, the curved dg algebra $\Omega^{\bullet}(M,\text{End}(E))$ from Example \ref{ex1section1}
is a  particular example of such a deformation.
In general, a related example of a curved dg algebra that is a deformation of a dg algebra
may  be obtained as follows.

\begin{example}\label{ex3section1}
Endow the bundle $E$ with an inner product and embed it into a trivial bundle $\widetilde{E}$
of rank $l$ in order to obtain a matrix idempotent $P$ over the algebra of smooth functions on $M$
(i.e., $P$ is a section of $\text{End} \bigl(\widetilde{E}\bigr)$ with $P^2=P$,
that projects onto $E \subseteq \widetilde{E}$).
Now we can deform the dg algebra of matrix-valued forms
$\Omega^{\bullet}(M,\textbf{M}_{l}(\textbf{k}))$
$\cong$
$\Omega^{\bullet}(M) \otimes \textbf{M}_{l}(\textbf{k})$
by the degree 1 element $\gamma=PdP$ as above.
\end{example}

Next we describe a well-known alternative description of  curved $A_{\infty}$-algebras.
Let $\text{T}(V)=\textbf{k}\oplus V \oplus V^{\otimes{2}}\oplus \dots$
be the {\it full} tensor coalgebra of the graded vector space $V$.
If $V=\bigoplus_{p}V_{p}$ its {\bf suspension} $\text{s}V$ is defined by $(\text{s}V)_{p}=V_{p-1}$.

\begin{proposition}\label{altdefa}
There is a natural bijection between curved $A_{\infty}$-structures on $V$ and square zero coderivations
of degree $1$ on $\text{T}(\text{s}V)$.
\end{proposition}

The reader is referred to \cite[Proposition 1.2]{GJ} for the proof and the relevant definitions.
This proposition allows us to make the following

\begin{definition} A homomorphism between two curved $A_{\infty}$-algebras $V$ and $V'$ is a map
$F:\text{T}(\text{s}V) \rightarrow \text{T}(\text{s}V')$ of dg coalgebras.
\end{definition}

More explicitly, such a homomorphism $F$ is given by a family of linear maps
$F_{k}:V^{\otimes{k}}\rightarrow V'$ of degree $1-k$, $k=0,1,2,\ldots$
satisfying
\be{ainfhom1}
F_{1} m_{0}^{V} \, = \, m_{0}^{V'}\,,
\ee
and for every $n>0$ the identity (cf. \cite{N})
$$
\sum_{r+s+t \, = \, n}\!(-1)^{r+st} \, F_{r+t+1} \,\, (\textbf{1}^{\otimes{r}}\otimes m^V_{s} \otimes \textbf{1}^{\otimes{t}})
$$
\be{ainmorph}
\, = \,
\hspace{-7pt}
\smtwo{1 \, \leqslant \, q \, \leqslant \, n}{i_{1}+ \ldots + i_{q} \, = \, n}
\hspace{-9pt}
(-1)^{w} \, m^{V'}_{q} \,\, (F_{i_{1}}\otimes \ldots \otimes F_{i_{q}})\,,
\ee
where
$w=\sum_{l \, = \, 1}^{q-1}(q-\ell)(i_{\ell}-1)$.
Note that $F_{0}$ necessarily is 0. The morphism $F$ is called {\bf strict} if  $F_{k}=0$ for every $k>1$.

Let $F=\{F_{k}\}: V \rightarrow V'$ and $G=\{G_{k}\}: V' \rightarrow V''$ be morphisms between curved $A_{\infty}$-algebras. Then the $p$-th
component of their composition is given by the formula (cf. \cite{N})

\be{morphcomp}
(G  F)_{p}=\hspace{-7pt}\smtwo{1 \, \leqslant \, q \, \leqslant \, p}{i_{1}+ \ldots + i_{q} \, = \, p}
\hspace{-9pt}
(-1)^{w} \, G_{q} \,\, (F_{i_{1}}\otimes \ldots \otimes F_{i_{q}})\, .
\ee

\medskip

We proceed to define curved $L_{\infty}$-algebras following the sign conventions of \cite{LM}.
Given graded variables $x_{1},x_{2},\dots, x_{n}$ and a permutation $\sigma \in \textbf{S}_{n}$,
let $\epsilon(\sigma;x_{1},\ldots, x_{n})$ be the Koszul sign obtained when passing from
$x_{1},\ldots, x_{n}$ to $x_{\sigma(1)},\ldots,x_{\sigma(n)}$.
Set
\be{chi}
\chi(\sigma)=\text{sgn}(\sigma)\,\epsilon(\sigma;x_{1},\ldots, x_{n})\,.
\ee
We say that $\sigma \in \textbf{S}_{n}$ is a $(k,n-k)$--{\bf unshuffle},
$0\leqslant k \leqslant n$, if  $\sigma(1)<\cdots<\sigma(k)$ and $\sigma(k+1)<\cdots<\sigma(n)$.

\begin{definition}\label{curvedLinf}
A {\bf curved $L_{\infty}$-structure} on a graded vector space $V$ is given by a family of linear maps
$l_{k}:V^{\otimes{k}}\rightarrow V$ of degree $2-k$, $k=0,1,2,\ldots$ that are
{\em graded}
antisymmetric in the sense that
\[
l_{k}(x_{\sigma(1)},\ldots,x_{\sigma(k)})=\chi({\sigma}) \, l_{k}(x_{1},\ldots, x_{k})
\]
for all $\sigma \in \textbf{S}_{n}$ and $x_{1},\ldots, x_{k}\in V$,
and moreover satisfy the following generalized Jacobi identities for every $n\geqslant 0$
\be{Liide}
\sum_{k \, = \, 0}^{n}(-1)^{k(n-k)}\sum_{\sigma}
\chi({\sigma}) \, l_{n-k+1}(l_{k}(x_{\sigma(1)},\ldots, x_{\sigma(k)}),x_{\sigma(k+1)},\ldots,x_{\sigma(n)}) \, = \, 0 \,,
\ee
where the summation is taken over all $(k,n-k)$ unshuffles $\sigma$.
\end{definition}

In the case when the curvature $l_{0}:k \rightarrow V_{2}$ is $0$, we obtain the definition
of an $L_{\infty}$-algebra (cf. \cite{LS}). When the maps $l_{k}$ are equal to $0$ for all $k>2$ we obtain the notion of {\bf curved dg Lie algebra}.

\begin{example}
Let $L$ be (an ordinary) dg Lie algebra with differential $d$ and let $\gamma$
be a fixed element of degree $1$ in $L$.
Then it is straightforward to check that we can deform $L$ to a curved dg Lie algebra as follows.
Set
\[
l_{0}=d\gamma+\frac{1}{2}[\gamma , \gamma]\,,
\]
\[
l_{1}=d+\bigl[\gamma,\cdot\bigr]
\]
and take $l_{2}$ to be commutator $\bigl[\cdot,\cdot\bigr]$ in $L$.
\end{example}

\begin{example}\label{cdglex}
Let $G$ be Lie group with Lie algebra $\mathfrak{g}$ and let $P$ be a smooth principal $G$-bundle.
The space $\Omega^{\bullet}(P,\mathfrak{g})$ of $\mathfrak{g}$-valued differential forms on the total space of
$P$ has a natural structure of a dg Lie algebra. One can use a connection on $P$, i.e. an invariant $1$-form in
$\Omega^{\bullet}(P,\mathfrak{g})$, to deform this structure to a curved dg Lie algebra as above. A related and perhaps more interesting example
can be obtained by considering the subalgebra of basic (with respect to the action of $G$) forms on $P$ instead of $\Omega^{\bullet}(P,\mathfrak{g})$.
\end{example}

As in the $A_{\infty}$ case, one has an alternative description of  curved $L_{\infty}$-structures
in terms of coderivations.
Let $\text{S}^{k}(V)$ denote the $k$-th (graded) symmetric power of the graded vector space $V$
and let $\text{S}(V)=\textbf{k}\oplus V \oplus \text{S}^{2}(V) \oplus \dots$
be the {\it full} symmetric coalgebra of  $V$.

\begin{proposition} There is a natural bijection between curved $L_{\infty}$-structures
on $V$ and square zero coderivations of degree $1$ on $\text{S}(\text{s}V)$.
\end{proposition}

It was proved in \cite{LS} that there is a bijection
between (ordinary) $L_{\infty}$-structures
on $V$ and square zero coderivations of degree $1$ on the {em reduced} symmetric algebra on  $\text{s}V$. The proof in the curved case is completely analogous; one simply has to replace the reduced symmetric algebra with the full symmetric algebra in order to incorporate
the curvature $l_{0}$.
As above, this proposition allows us to define homomorphisms of curved $L_{\infty}$-algebras.

\subsection{Curved $A_{\infty}$- and $L_{\infty}$-structures as deformations}

In this subsection we describe a method of perturbing $A_{\infty}$- and $L_{\infty}$-algebras to curved ones,
generalizing the examples already given in the previous subsection.
We first treat the (slightly easier) case of $L_{\infty}$-algebras.

Let $L$ be an $L_{\infty}$-algebra with maps $\ell_{k}:L^{\otimes{k}}\rightarrow L$ and let $\gamma$ be an element
in $L$ of degree $1$.
For  $x_{1},\ldots,x_{k} \in L$ consider the (formal) infinite series:
\begin{align}\label{deforml}
\ell_{0}^{\gamma} \, = & \ \sum_{k \, = \, 1}^{\infty}\frac{1}{k!}\,\ell_{k}(\gamma,\gamma,\dots, \gamma) \,,
\\
\ell_{n}^{\gamma}(x_{1},\ldots,x_{n}) \, = & \
\ell_{n}(x_{1},\ldots,x_{n})
+\sum_{k \, = \, 1}^{\infty}\frac{1}{k!}\,\ell_{k+n}(\gamma,\ldots ,\gamma,x_{1},\ldots,x_{n})
\,, \quad n \,\geqslant\, 1.
\nonumber
\end{align}
We have the following generalization of Proposition 4.4 and Lemma 4.5 in \cite{G}.

\begin{proposition}\label{linfdeform}
Let $L$ be an $L_{\infty}$-algebra which is a Banach space and assume that all maps $\ell_{k}$ are continuous
{\rm (}with respect to the induced norm on the tensor powers of $L${\rm )}.
Suppose also that the series (\ref{deforml}) are point-wise norm convergent.
Then these series define a curved $L_{\infty}$-structure on $L$.
\end{proposition}

\begin{proof}
Let us introduce the abbreviation $\ell_{k}(\gamma^{\wedge (k-i)},x_{1},\ldots,x_{i})$ for
$\ell_{k}(\gamma,$ $\ldots,$ $\gamma,$ $x_{1},$ $\ldots,$ $x_{i})$
($\gamma$ appears $k-i$ times).
A direct computation, via rearranging the terms in the sums, shows that
\begin{align*}
\sum_{k \, = \, 0}^{n} \, & (-1)^{k(n-k)}
\sum_{\sigma} \, \chi({\sigma}) \,
\ell_{n-k+1}^{\gamma}\Bigl(\ell_{k}^{\gamma}\bigl(x_{\sigma(1)},\ldots,
x_{\sigma(k)}\bigr),x_{\sigma(k+1)},\ldots,x_{\sigma(n)}\Bigr)
\\
= \, & \sum_{k \, = \, 0}^{n} \, \sum_{\sigma} \, (-1)^{k(n-k)} \, \chi({\sigma})
\sum_{p \, = \, 0}^{\infty}
\, \frac{1}{p!} \,
\sum_{s \, = \, 0}^{p} \, \frac{p!}{s!(p-s)!}
\\
& \times \ell_{n-k+p+1}\Bigl(\gamma^{\wedge s},\ell_{k+p-s}\bigl(\gamma^{\wedge (p-s)},
x_{\sigma(1)},\ldots, x_{\sigma(k)}\bigr),x_{\sigma(k+1)},\ldots,x_{\sigma(n)}\Bigr)
\\
= \, &
\sum_{p \, = \, 0}^{\infty} \, \frac{1}{p!} \,
\sum_{k \, = \, 0}^{p+n} \, (-1)^{k(p+n-k)} \,
\sum_{\varrho} \, \chi({\varrho}) \,
\\
& \times
\ell_{p+n-k+1}\Bigl(
\ell_{k}\bigl(y_{\varrho(1)},\ldots, y_{\varrho(k)}\bigr),y_{\varrho(k+1)},\ldots,y_{\varrho(p+n)}\Bigr)
\, = \, 0 \,,
\end{align*}
where $y_1=\gamma,\dots,y_p=\gamma,y_{p+1}=x_1,\dots,y_{p+n}=x_n$ and the sum
$\mathop{\sum\limits_{\varrho}}$ is over all \newline $(k,$ $p$ $+$ $n$ $-$ $k)$--unshuffles $\varrho$.
\end{proof}

\begin{remark}
Note that we can similarly deform a curved $L_{\infty}$-algebra by allowing
$k$ to run from $0$ to $\infty$ in the first of Eqs.~(\ref{deforml}). Proposition \ref{linfdeform} with its proof remain valid in this more general setting.
\end{remark}

The process of deformation just described is functorial in the following sense.
Let $L$ and $M$ be two curved $L_{\infty}$-algebras, let $\gamma \in L^{1}$ and
$F= \{F_k\} :L \rightarrow M$
be a morphism of curved $L_{\infty}$-algebras. Suppose that the series
\[\gamma ' = \sum_{k \, = \, 1}^{\infty}
\frac{1}{k!}\, F_{k}(\gamma,\ldots ,\gamma)\]
is norm convergent. Assume further that we can deform $L$ and $M$ by $\gamma$ and $\gamma '$ respectively,
as in Proposition \ref{linfdeform}, to curved $L_{\infty}$-algebras that we denote by $L^{\gamma}$ and
$M^{\gamma '}$, respectively.
Now for every $n\geqslant 1$ consider the series
\be{linfmorph}
F_{n}^{\gamma}(x_{1},\ldots,x_{n}) \, = \,
F_{n}(x_{1},\ldots,x_{n})+\sum_{k \, = \, 1}^{\infty}
\frac{1}{k!} \, F_{n+k}(\gamma,\ldots ,\gamma,x_{1},\ldots,x_{n})\,.
\ee
By a computation similar to the one in the proof of Proposition \ref{linfdeform}, one can show that
the following holds.

\begin{proposition}\label{linfmorphdeform}
Assume that all maps $F_{k}$ representing the morphism
$F :L \rightarrow M$ are continuous and that the series (\ref{linfmorph}) are point-wise norm convergent.
Then these series define a morphism of curved $L_{\infty}$-algebras
\begin{equation}
F^{\gamma}:L^{\gamma} \rightarrow M^{\gamma '} \,.
\tag*{$\Box$}
\end{equation}
\end{proposition}

Curved $A_{\infty}$-algebras can be obtained as deformations of ordinary $A_{\infty}$-algebras
in an analogous fashion.
More generally,
let $A$ be a curved $A_{\infty}$-algebra with maps $m_{k}:A^{\otimes{k}}\rightarrow A$
($k=0,1,\dots$)
and let $\gamma$ be in $A_{1}$.
For  $x_{1},\ldots,x_{k} \in A$ consider the infinite series
\begin{align}\label{deforma}
m_{0}^{\gamma} \, & =  \, \sum_{k \, = \, 0}^{\infty}m_{k}(\gamma,\gamma,\dots, \gamma) \,,
\\
m_{n}^{\gamma}(x_{1},\ldots,x_{n}) \, & =  \,
\sum_{k \,  = \, 0}^{\infty}\sum_{\sigma} \, \chi(\sigma) \,
m_{n+k}\bigl(y_{\sigma(1)},\ldots ,y_{\sigma(k)},y_{\sigma(k+1)},\ldots,y_{\sigma(k+n)}\bigr) \,,
\nonumber
\end{align}
where
$y_1=\gamma,\dots,y_k=\gamma,y_{k+1}=x_1,\dots,y_{k+n}=x_n$,
the summation is taken over all permutations
$\sigma \in \textbf{S}_{k+n}$ that are $(k,k+n)$-unshuffles,
and $\chi(\sigma)$ is the sign appearing in (\ref{Liide}).

\begin{proposition}\label{ainfdeform}
Let $A$ be a curved $A_{\infty}$-algebra that is a Banach space.
Assume that the maps $\{m_{k}\}_{k=0}^{\infty}$ are continuous and that the series (\ref{deforma}) are point-wise norm convergent.
Then these series define a curved $A_{\infty}$-structure on $A$.
\end{proposition}

\begin{proof}
We check that Eq. (\ref{Ainfide}) holds and obtain
\begin{align}\label{calc1}
\sum_{r+s+t \, = \, n} & (-1)^{r+st} \, m^{\gamma}_{r+t+1}
\Bigl(x_1,\dots,x_r,m^{\gamma}_{s} \bigl(x_{r+1},\dots,x_{r+s}\bigr),x_{r+s+1},\dots,x_n\Bigr)
\nonumber \\
= & \, \sum_{r+s+t \, = \, n} \,
\sum_{p \, = \, 0}^{r+t+1} \, \sum_{q \, = \, 0}^{s} \,
(-1)^{r+st} \, \chi_{p,q,r,s} \,
\nonumber \\ & \times
m_{p+r+t+1}
\Bigl(
\mathop{\dots}\limits^{\mathop{\vee}\limits^{\hspace{-0.5pt}\gamma\hspace{0.5pt}}},
x_1,\mathop{\dots}\limits^{\mathop{\vee}\limits^{\hspace{-0.5pt}\gamma\hspace{0.5pt}}},
x_r,\mathop{\dots}\limits^{\mathop{\vee}\limits^{\hspace{-0.5pt}\gamma\hspace{0.5pt}}},
m_{q+s} \bigl(\mathop{\dots}\limits^{\mathop{\vee}\limits^{\hspace{-0.5pt}\gamma\hspace{0.5pt}}},
x_{r+1},\mathop{\dots}\limits^{\mathop{\vee}\limits^{\hspace{-0.5pt}\gamma\hspace{0.5pt}}},
x_{r+s},\mathop{\dots}\limits^{\mathop{\vee}\limits^{\hspace{-0.5pt}\gamma\hspace{0.5pt}}}\bigr),
\nonumber \\
& \hspace{59pt}
\mathop{\dots}\limits^{\mathop{\vee}\limits^{\hspace{-0.5pt}\gamma\hspace{0.5pt}}},
x_{r+s+1},\mathop{\dots}\limits^{\mathop{\vee}\limits^{\hspace{-0.5pt}\gamma\hspace{0.5pt}}},
x_n,\mathop{\dots}\limits^{\mathop{\vee}\limits^{\hspace{-0.5pt}\gamma\hspace{0.5pt}}}\Bigr)
\,,
\end{align}
where $\mathop{\vee}\limits^{\hspace{-0.5pt}\gamma\hspace{0.5pt}}$ stands for insertions of $\gamma$'s,
 we have inserted $p$ $\gamma$'s outside $m_s$ and $q$ $\gamma$'s inside,
and $\chi_{p,q,r,s}$ is the resulting sign factor.
Then the right hand side of Eq.~(\ref{calc1}) reads
\begin{align*}
&
\sum_{k \, = \, 0}^{\infty} \, \sum_{\sigma} \, \chi (\sigma) \! \sum_{r+s+t \, = \, k+n} (-1)^{r+st} \,
\\ &
\times
m_{r+t+1}
\Bigl(y_{\sigma(1)},\dots,y_{\sigma(r)},m_{s} \bigl(y_{\sigma(r+1)},\dots,y_{\sigma(r+s)}\bigr),
y_{\sigma(r+s+1)},\dots,y_{\sigma(k+n)}\Bigr)
\, = \, 0 \,,
\end{align*}
where the same notation as in Eq. (\ref{deforma}) is used.
\end{proof}

\medskip

As in the $L_{\infty}$ case, we have the following functoriality result.
Let $A$ and $B$ be two curved $A_{\infty}$-algebras, let $\gamma \in A^{1}$ and
$F=\{F_k\}:A \rightarrow B$ be a morphism of curved $A_{\infty}$-algebras. Assume that the series
\[\gamma ' = \sum_{k \, = \, 1}^{\infty}
F_{k}(\gamma,\ldots ,\gamma)  \]
is norm convergent. Suppose further that we can deform $A$ and $B$ by $\gamma$ and $\gamma '$ respectively,
as in Proposition \ref{ainfdeform}, to curved $A_{\infty}$-algebras that we denote by
$A^{\gamma}$ and $B^{\gamma '}$.
Now for every $n\geqslant 1$ consider the series
\be{ainfmorph}
F_{n}^{\gamma}(x_{1},\ldots,x_{n}) \, = \,
\sum_{k \, = \, 0}^{\infty}\sum _{\sigma}\chi(\sigma)F_{n+k}(y_{\sigma(1)},\ldots,y_{\sigma(k+n)})
\,,
\ee
where again
$y_1=\gamma,\dots,y_k=\gamma,y_{k+1}=x_1,\dots,y_{k+n}=x_n$
and the summation is taken over all permutations
$\sigma \in \textbf{S}_{k+n}$ that are $(k,k+n)$-unshuffles.

\begin{proposition}\label{ainfmorphdeform}
Assume that the maps $\{F_{k}\}_{k \, = \, 1}^{\infty}$ representing the morphism
$F=A \rightarrow B$ are continuous and that the series (\ref{ainfmorph}) are point-wise norm convergent.
Then these series define a morphism of curved $A_{\infty}$-algebras
\begin{equation}
F^{\gamma}:A^{\gamma} \rightarrow A^{\gamma '}\,.
\tag*{$\Box$}
\end{equation}
\end{proposition}

\subsection{Examples}
Let $X$ be a locally finite polyhedron.
We briefly recall, following \cite{CG}, how the complex $C^{\bullet}(X)$
of $\textbf{k}$-valued oriented simplicial cochains on $X$ can be endowed with an $A_{\infty}$-structure
(which is in fact a $C_{\infty}$-structure).
For alternative  constructions of such a structure the reader is referred to the Appendix of \cite{TZ}
and to \cite{DMS}.

We recall the definition of the complex of the piece-wise smooth differential forms on $X$ introduced
by H. Whitney in \cite{Wh} (see also \cite[Chapter 2]{Du} and \cite{BG}). A piece-wise smooth differential $n$-form on $X$ is a collection of forms $\omega=\{\omega_{\sigma}\}$, one for
each simplex $\sigma$ in $X$, such that $\omega_{\sigma}$ is a smooth $n$-form on $\sigma$ and for every inclusion of simplices $i:\sigma_{1} \hookrightarrow \sigma_{2}$ we have $i^{*}\omega_{\sigma_{2}}=\omega_{\sigma_{1}}$. We denote the space of all such collections of forms by $\Omega^{\bullet}(X)$. Since the exterior differential commutes with pull-backs we obtain a well-defined differential on  $\Omega^{\bullet}(X)$.

 The following theorem is due to J. Dupont \cite [Theorem 2.16] {Du}.

\begin{theorem}\label{contraction}
There exists a chain contraction  from
$\Omega^{\bullet}(X)$
to $C^{\bullet}(X)$.
Name\-ly, there exist chain maps $R: \Omega^{\bullet}(X) \rightarrow C^{\bullet}(X)$ and
$W: C^{\bullet}(X) \rightarrow \Omega^{\bullet}(X)$, and a linear map
$H:\Omega^{\bullet}(X) \rightarrow \Omega^{\bullet-1}(X)$ satisfying the relations
\[
R W \, = \, \textbf{1}_{C^{\bullet}(X)}
\]
\begin{equation}
W R-\textbf{1}_{\Omega^{\bullet}(X)} \, = \, dH+Hd \,.
\tag*{$\Box$}
\end{equation}
\end{theorem}

The map $R$ is given by integration over simplices. The map $W$ can be described as a generalized linear interpolation; its image consist of piece-wise linear forms (see \cite{Wh}).
We shall refer to $R$ (respectively $W$) as the de Rham (respectively Whitney) map.

Note that the complex $\Omega^{\bullet}(X)$ is a dg algebra with respect to the exterior product of forms.
Thus, we can transfer this particular $A_{\infty}$-structure to an $A_{\infty}$-structure on the complex
$C^{\bullet}(X)$ along the contraction given from Theorem \ref{contraction},
using the sum over rooted planar tree formulas of Kontsevich and Soibelman \cite{KS}.
In Sect. 3 we will review these formulas in some detail.

Using a piece-wise smooth metric on $X$ one can define an $L_{2}$-norm on forms in
$\Omega^{\bullet}(X)$
and, using the Whitney embedding $W$, also on cochains in $C^{\bullet}(X)$. We shall denote this norm by $\|.\|$ in what follows. Let $\{m_{k}\}_{k \, = \, 1}^{\infty}$ be the maps defining the
$A_{\infty}$-structure on $C^{\bullet}(X)$ just described.

\begin{lemma}\label{mkestimate} For all $k>1$ and all $x_{1},\ldots, x_{k}\in C^{\bullet}(X)$ one has
\[
\|m_{k}(x_{1},\ldots, x_{k})\| \, \leqslant \, c_{k-1} \, E^{k}\|x_{1}\|\cdots\|x_{k}\| \,,
\]
where $c_{k}=\frac{(2k)!}{(k+1)!k!}$ is the $k$-th Catalan number and $E$ is a constant depending only
on $X$ and the norm $\|.\|$.
\end{lemma}

\begin{proof} Observe that the expression for $m_{k}$ given in \cite{KS} in this case reduces to a sum
over all completely binary trees with $k$
tails (see Sect. 3 for the relevant definitions)
and the number of these trees is exactly
$c_{k-1}$.
Now the estimate follows since the exterior product and the homotopy $H$ are bounded maps in the $L_{2}$-norm.
\end{proof}

Observe that we can extend all maps involved in Theorem \ref{contraction} entry-wise to matrix-valued forms and matrix-valued cochains.
We can apply Theorem \ref{contraction} and Lemma \ref{mkestimate} to each matrix entry and conclude that these two statements remain valid in the matrix setting. Thus we obtain an $A_{\infty}$-structure on the complex of matrix-valued cochains
$C^{\bullet}(X,\textbf{M}_{l}(\textbf{k}))$.
We shall deform this to a curved $A_{\infty}$-structure using a degree $1$ element as in Proposition
\ref{ainfdeform}.

Let  $E$ be a (continuous) complex vector bundle over $X$.
As in Example \ref{ex3section1}, there is a continuous matrix idempotent
representing $E$. Observe that $\Omega^{0}(X)$, the algebra of piece-wise smooth functions on $X$, is a dense (by the Stone-Weierstrass theorem) subalgebra of the Banach algebra $C(X)$ of continuous functions on $X$,
which is closed under holomorphic functional calculus. It follows (see e.g. \cite[Chapter III, Appendix C]{Co})
that the inclusion  $\Omega^{0}(X)\hookrightarrow C(X)$ induces isomorphism of $K$-theory groups: $ K_{0}(\Omega^{0}(X))\cong K_{0}(C(X))$.
In other words, there exists a piece-wise smooth idempotent
$P \in \Omega^{0}(X,\textbf{M}_{l}(\textbf{k}))$
representing the K-theory class of $E$.

\begin{proposition}\label{exadef}
For every fine enough subdivision $\widetilde{X}$ of $X$, the series (\ref{deforma}) with $\gamma=R(PdP)$
{\rm (}where $R$ is the de Rham map for the complex $\widetilde{X}${\rm )}
define a curved $A_{\infty}$-structure on $C^{\bullet}(\widetilde{X},\textbf{M}_{l}(\textbf{k}))$.
\end{proposition}

\begin{proof}
Since the analytical hypotheses of Proposition \ref{ainfdeform} are not satisfied in this setting,
we shall show that the series (\ref{deforma}) are locally convergent in an appropriate sense.
To this end, observe that the maps $\{m_{k}\}$ are quasi-local in the sense that the value of
the cochain $m_{k}(x_{1},\ldots, x_{k})$ at a simplex $\Delta$ depends only on the values of
$x_{1},\ldots, x_{k}$ restricted to  $\textbf{St}(\Delta)$, the star of $\Delta$.
Thus Lemma \ref{mkestimate} gives
\[
\|m_{k}(x_{1},\ldots, x_{k})|_{\Delta}\| \, \leqslant \,
c_{k-1} \,
E^{k}\,\|x_{1}|_{\textbf{St}\Delta}\|\cdots\|x_{k}|_{\textbf{St}\Delta}\|\,.
\]
Using the elementary estimate  $c_{k}<4^{k}$ we obtain
\[
\|m_{p+k}(\gamma,\ldots, \gamma, x_{1},\ldots, x_{k})\| \, \leqslant \,
(4E)^{p+k}\,\|\gamma\|^{p}\,\|x_{1}\|\cdots\|x_{k}\|\,.
\]
Thus we see that if we choose a fine enough subdivision $\widetilde{X}$ of $X$
so that the norm of $\gamma$ restricted to the star of each simplex in
$\widetilde{X}$ is less than $1/(4E)$, the series (\ref{deforma}) will locally
(i.e. at each simplex) be dominated by convergent geometric series and hence will be convergent.
 \end{proof}

Finally, we briefly sketch how one can obtain an example of a curved $L_{\infty}$-structure based on
Example \ref{cdglex}.
Let $G$ be Lie group with Lie algebra $\mathfrak{g}$, $P$ a smooth principal $G$-bundle,
and $X$ a triangulation of  the total space of $P$.
Using Theorem \ref{contraction} and the main result of \cite{H2} one can transfer
the dg Lie algebra structure on
$\Omega^{\bullet}(P,\mathfrak{g})$
to an
$L_{\infty}$-structure on the complex of $\mathfrak{g}$-valued cochains $C^{\bullet}(X,\mathfrak{g})$.
Let $\alpha$ be connection $1$-form on $P$. Then one has the following analogue of Proposition \ref{exadef}.

\begin{proposition}\label{exadef2}
For every fine enough subdivision $\widetilde{X}$ of $X$, the series (\ref{deforml}) with $\gamma=R(\alpha)$
{\rm (}where $R$ is the de Rham map for the complex $\widetilde{X}${\rm )} define a curved
$L_{\infty}$-structure on
$C^{\bullet}(\widetilde{X},\mathfrak{g})$. \hfill $\Box$
\end{proposition}

It is clear however that this curved $L_{\infty}$-structure on $X$ does not reflect the $G$ action on $P$; perhaps a more relevant example of a
curved $L_{\infty}$-structure can be obtained by considering a $G$-equivariant triangulation of $P$ and $G$-invariant cochains.


\section{{Transferring curved dg and Lie dg structures that are deformations}}

\subsection{Review of homological perturbation theory}
In this section we recall the classical homological perturbation lemma (\cite{Gu}) and show how to use it  to transfer curved dg and Lie dg algebra structures that are deformations of ordinary ones. Given a chain complex $(C,d)$, we say that a map $\delta: C^{\bullet}\rightarrow C^{\bullet +1}$ is a {\bf perturbation} of $d$ if $(d+\delta)^{2}=0$.

\begin{definition}\label{contr} Assume that we are given two chain complexes $(C_{1},d_{1})$ and $(C_{2},d_{2})$ and chain maps $p: C_{1} \rightarrow C_{2}$ and
$i: C_{2} \rightarrow C_{1}$. We say that the pair $(p,i)$ is a {\bf chain contraction} of $(C_{1},d_{1})$ to $(C_{2},d_{2})$ if $pi=\textbf{1}_{C_{1}}$ and there exists a homotopy between $ip$ and the identity map on $C_{1}$, i.e. a map $H:C_{1}^{\bullet}\rightarrow C_{1}^{\bullet -1}$ such that $ip-\textbf{1}_{C_{1}}=d_{1}H+Hd_{1}$. We say that  the contraction $(p,i)$ is  {\bf special} if the following so called annihilation conditions hold.
\[Hi=0, \quad pH=0, \quad H^{2}=0.\]
\end{definition}
It is shown in  \cite{G} that the chain contraction described in Theorem \ref{contraction} is special.

The first part of the following statement is known as the main homological perturbation lemma. The second part is proved in \cite[Section 2]{HK}.
\begin{theorem}\label{pertlemma}
$(a)$ Let $(C_{1},d_{1})$ be a chain complex and $\delta_{1}$ be a perturbation of $d_{1}$.
Suppose that we are given a special chain contraction $(p,i,H)$ between $(C_{1},d_{1})$
and a chain complex $(C_{2},d_{2})$.
Assume further that $\textbf{1}_{C_{1}}-\delta_{1} H$ is invertible and set
$\Sigma = (\textbf{1}_{C_{1}}-\delta_{1} H)^{-1}\delta_{1}$.
Then $\delta_{2}=p\Sigma i$ is a perturbation of $d_{2}$ and the formulas
\be{mainpert}
\widetilde{i} \, = \, i+H\Sigma i, \quad \widetilde{p}=p+p\Sigma H, \quad \widetilde{H}=H+H\Sigma H
\ee
define a special chain contraction between $(C_{1},d_{1}+\delta_{1})$ and $(C_{2}, d_{2}+\delta_{2})$.

$(b)$ Assume, in addition to the hypotheses in part $(a)$, that $(C_{1},d_{1})$
is dg {\rm (}co{\rm )}al\-ge\-b\-ra,
$\delta_{1}$ is a {\rm (}co{\rm )}derivation, and $p$ and $i$ are {\rm (}co{\rm )}algebra homorphisms.
Then $\delta_{2}$ is a coderivation and the maps $\widetilde{p}$ and $\widetilde{i}$ are
{\rm (}co{\rm )}algebra homomorphisms.
\end{theorem}

\subsection{Transferring curved dg and Lie dg algebra structures that are deformations of ordinary ones}

\begin{proposition}\label{transa}
Let $A$ be dg algebra equipped with a Banach norm such that the multiplication map is continuous.
Let $\gamma$  be a degree 1 element of $A$  and put $\widehat{\gamma}(x)=[\gamma, x]$ for $x \in A$. Let $A^{\gamma}$
denote the curved dg algebra obtained from $A$ via $\gamma$ as in Example \ref{ex1section1}.
Let $(p,i,H)$ be a special contraction from $A$ to a chain complex $B$ which is a Banach space,
such that $H$ is continuous and $p$ and $i$ are continuous with norms not exceeding 1.
Assume further that $\|\widehat{\gamma}\| \, \|H\|<1$. Then there exists a curved $A_{\infty}$-structure on $B$
and morphisms of curved $A_{\infty}$-algebras $P:A^{\gamma} \rightarrow B$ and $I:B \rightarrow A^{\gamma}$.
\end{proposition}

\begin{proof}
According to Proposition \ref{altdefa}, the curved dg algebra structure on $A$ obtained via $\gamma$
is equivalent to a coderivation $D_{A}$ on $T(sA)$ such that $D_{A}^{2}=0$.
Let $T(d_{A})$ denote the usual extension of the differential $d_{A}$ on $A$ to $T(sA)$.
It follows from the proof of Proposition \ref{altdefa}
that we can write $D_{A}$ as a perturbation of $T(d_{A})$: $D_{A}=T(d_{A})+\delta_{A}$.
Since $T(d_{A})$ and $D_{A}$ are coderivations, so is $\delta_{A}$.

Now let us  define maps $T(p):T(sA) \rightarrow T(sB)$,
$T(i):T(sB) \rightarrow T(sA)$ and $T(H):T(sA) \rightarrow T(sA)$ by setting for every $n \geqslant 1$:
\be{tensortrick}
(T(p))_{n}=p^{\otimes n}, \quad (T(i))_{n}=i^{\otimes n}, \quad
(T(H))_{n}=\sum _{k+j+1=n}\textbf{1}^{\otimes k}\otimes H \otimes (ip)^{\otimes j}.
\ee
(The last formula is a particular case of the so-called tensor trick, see \cite{H1} and the references therein.) It is easy to check that (\ref{tensortrick}) defines a special contraction from $T(sA)$ to $T(sB)$ and that $T(p)$ and $T(i)$ are coalgebra morphisms. The assumption $\|\widehat{\gamma}\| \, \|H\|<1$ implies that the series
\[
\sum _{k=0}^{\infty}(\delta_{A} T(H))^{k}
\]
is dominated by a convergent geometric series. Thus $1-\delta_{A} T(H)$ is invertible and we can apply
Theorem \ref{pertlemma} to the contraction (\ref{tensortrick}).
As a result, we obtain a coderivation $\delta_{B}$ on $T(sB)$, which is a perturbation of $T(d_{B})$ and hence
defines a curved $A_{\infty}$-structure on $B$, and also curved $A_{\infty}$-algebra morphisms between
$A^{\gamma}$ and $B$ given by (\ref{mainpert}).
\end{proof}

We note that the case of an ordinary dg algebra in the proposition above has already been considered
in \cite{H1} where it is also shown that the formulas for the transferred $A_{\infty}$-structure obtained
from the perturbation lemma coincide with the sum of over rooted tree formulas obtained in \cite{KS}.
We will now generalize these formulas to the case of a transfer of a curved dg structure.

Recall that a rooted planar tree has a natural orientation of its edges, which we use to define the
{\bf valency} of a vertex of such a tree to be the number of edges going out from the vertex.
A {\bf tail} is a vertex of valency zero.
A {\bf completely binary} tree is rooted planar tree whose root has valency 1
and all other (non-tail) vertices have valency 2. Let us denote by $\Upsilon _{n}$
the set of all completely binary rooted planar trees with $n$ tails.
Define further $\Upsilon _{n}^{dec}$ to be the set whose elements consist of a tree in $\Upsilon _{n}$
and a sequence of arbitrary non-negative integers, one for each edge of the tree.
Assuming the hypotheses of Proposition \ref{transa}, we can think of an element in $\Upsilon _{n}^{dec}$
as a tree in $\Upsilon _{n}$ decorated according to the following rules:
\begin{LIST}{25pt}
\item[(1)] Assign to the root the map $p$ and to each tail the map $i$.
\item[(2)] Assign to each vertex of valency 2 the map $m_{2}^{A}$.
\item[(3)] Assign to the edge $e$ containing the root the map $(\widehat{\gamma}  H)^{k_{e}}$, $k_{e}=0,1,2,\ldots$
\item[(4)] Assign to each edge containing a tail $e$ the map $(H  \widehat{\gamma})^{k_{e}}$, $k_{e}=0,1,2,\ldots$
\item[(5)] Assign to each interior edge $e$ the map $H (\widehat{\gamma}H)^{k_{e}}$, $k_{e}=0,1,2,\ldots$
\end{LIST}
Now to each element $\tau$ in $\Upsilon _{n}^{dec}$ we can assign an operation $m_{\tau}: B^{\otimes n} \rightarrow B$ by composing the maps assigned to the
vertices and the edges of $\tau$ in the natural order. For example, to the tree
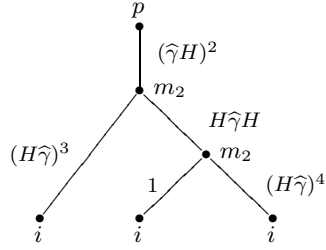
\[
\xygraph{
!{<0cm,0cm>;<2.21cm,0cm>:<0cm,0.85cm>::}
!{(0,2)}*{\Blt}="0"
!{(0,1)}*{\Blt}="1"
!{(0.4,0)}*{\Blt}="3"
!{(-0.6,-1)}*{\Blt}="6"
!{(0,-1)}*{\Blt}="7"
!{(0.8,-1)}*{\Blt}="9"
!{(-0.01,2.25)}*{\text{\small $p$}}
!{(-0.61,-1.25)}*{\text{\small $i$}}
!{(-0.01,-1.25)}*{\text{\small $i$}}
!{(0.79,-1.25)}*{\text{\small $i$}}
!{(0.18,1)}*{\text{\small $m_2$}}
!{(0.58,0)}*{\text{\small $m_2$}}
!{(0.29,1.59)}*{\text{\footnotesize $(\widehat{\gamma}  H)^{2}$}}
!{(0.57,0.52)}*{\text{\footnotesize $H  \widehat{\gamma}  H$}}
!{(0.94,-0.48)}*{\text{\footnotesize $(H  \widehat{\gamma})^{4}$}}
!{(-0.6,0)}*{\text{\footnotesize $(H  \widehat{\gamma})^{3}$}}
!{(0.08,-0.48)}*{\text{\footnotesize  $1$}}
"0"-"1" "1"-"3" "1"-"6" "3"-"7" "3"-"9"
}\]
there corresponds the map
\[
p  (\widehat{\gamma}  H)^{2}  m_{2}^{A} ((H  \widehat{\gamma})^{3} i \otimes
H  \widehat{\gamma}  H  m_{2}^{A}  ( i \otimes (H  \widehat{\gamma})^{4}) i)
.
\]

\medskip

\begin{proposition}\label{expform}
The transferred curved $A_{\infty}$-structure on $B$ obtained in Proposition \ref{transa} is given by
\be{formo}
m_{0}^{B}= \sum _{k=0}^{\infty}p  (\widehat{\gamma}  H)^{k} (m_{0}^{A}),\ee
\be{formone} m_{1}^{B}=p  d_{A}  i + \sum _{k=0}^{\infty} p  \widehat{\gamma} (\widehat{\gamma}  H)^{k}  i,\ee
\be{formn}
m_{n}^{B}=\sum_{\tau \in \Upsilon _{n}^{dec}} (\pm) m_{\tau}, \quad n>1.
\ee
\end{proposition}

\begin{proof} Formulas (\ref{formo}) and (\ref{formone}) follow easily from Proposition \ref{transa}. We shall verify (\ref{formn}) closely following \cite{Be2}. We first observe that the transfer formulas from Theorem \ref{pertlemma} can be rewritten as recursive formulas:
\[\begin{array}{rl}\delta_{B}=T(p)\delta_{A}I \hspace{50pt} & I=T(i)+T(H)\delta_{A}I\\
 P=T(p)+P\delta_{A}T(H) \hspace{50pt} & \widetilde{H}=T(H)+\widetilde{H}\delta_{A}T(H).\end{array}\]

Using this, we can derive an inductive formula for the components of the curved $A_{\infty}$-morphism $I:B
\rightarrow A^{\gamma}$. Proceeding as in the proof of Theorem 2.1 in \cite{Be2} we find that
\[ I_{1}=i+ H  \widehat{\gamma}  I_{1}.\]
Solving for $I_{1}$ and expanding $(1-H  \widehat{\gamma})^{-1}$ in geometric series we obtain
\be{eqIone}I_{1}=\sum_{k=0}^{\infty}(H  \widehat{\gamma})^{k}  i.
\ee
Next, for every $n>1$ we find
\[I_{n}=H  \widehat{\gamma} I_{n}+\sum_{l+s=n}(-1)^{l(s-1)} H  m_{2}^{A}(I_{l},I_{s}),\]
which implies the desired recursive formula:
\be{formIn}I_{n}=\sum _{k=0}^{\infty}\,\sum_{\,\, l+s=n}(-1)^{l(s-1)}H (\widehat{\gamma}  H)^{k} m_{2}^{A}(I_{l},I_{s}),
\ee

Similarly one shows that
\[m_{n}^{B}=p  \widehat{\gamma} I_{n}+\sum_{l+s=n} (-1)^{l(s-1)}  m_{2}^{A}(I_{l},I_{s}),\]
and substituting (\ref{formIn}) in the last expression gives
\be{forminm}m_{n}^{B}=\sum _{k=0}^{\infty}\,\sum_{\,\, l+s=n}(-1)^{l(s-1)}p (\widehat{\gamma}  H)^{k} m_{2}^{A}(I_{l},I_{s}).
\ee
Now substituting (\ref{eqIone}) and  (\ref{formIn}) into (\ref{forminm}) one obtains (see \cite{Be2} and \cite{M})
$$m_{n}^{B}=\sum_{\tau \in \Upsilon _{n}^{dec}} (-1)^{\epsilon(\tau)} m_{\tau},$$

where $\epsilon(\tau)$ is defined as follows. Let $v$ be a vertex of $\tau$ which is not a tail and not the root. Let $l_{v}$ be the
number of tails $t$ of $\tau$ such that the unique path from $t$ to the root of $\tau$ contains the first output edge of $v$ and let
$s_{v}$ be the
number of tails $t$ of $\tau$ such that the unique path from $t$ to the root of $\tau$ contains the second output edge of $v$. Then define
$\epsilon(\tau)=\sum_{v}l_{v}(s_{v}-1)$ (the sum is over all vertices $v$ which are neither a tail nor the root).

\end{proof}

\begin{remark}\label{formor} Clearly (\ref{formIn}) implies a sum over rooted trees formula for the components of $I$ similar to (\ref{formn}); the only difference is that we assign $H$ instead of $p$ to the roots of the trees. One can also derive a similar formula for the morphism $P:A^{\gamma} \rightarrow B$  as in \cite{Be2}. In this case each tree in the sum has $H$ assigned to {\em one} of its tails.
\end{remark}

We shall need the additional properties of the transferred curved  $A_{\infty}$-structure described in the following three lemmas in the next section.

\begin{lemma}\label{strictmor}
Let $(p_{k},i_{k},H_{k})$ be a special contraction from a dg algebra $A_{k}$ to a chain complex $B_{k}$, $k=1,2$, satisfying the hypotheses of Proposition \ref{transa}. Let $\gamma_{k} \in A_{k}$ be degree 1 elements and let $\psi: A_{1}^{\gamma_{1}} \rightarrow A_{2}^{\gamma_{2}}$ be a strict morphism of curved dg algebras such that
\be{strictcommute}
\psi  H_{1}=H_{2}  \psi
\ee
Consider $B_{k}$, $k=1,2$, with the transferred curved $A_{\infty}$-structure obtained in Proposition \ref{transa} and let $P_{k}=\{P_{k,l}\}_{l=1}^{\infty}$ and $I_{k}=\{I_{k,l}\}_{l=1}^{\infty}$ be the corresponding curved $A_{\infty}$-algebra morphisms.
Then the morphism $P_{2}  \psi  I_{1}: B_{1} \rightarrow B_{2}$ is also strict.
\end{lemma}

\begin{proof} We compute the components of $P_{2}  \psi  I_{1}$ using (\ref{morphcomp}) and Remark \ref{formor} and find that all terms contain compositions either of the form  $H_{2}  \psi  H_{1}$ or of the form $ p_{2}  \psi  H_{1}$. Now the claim follows from (\ref{strictcommute}), using the annihilation conditions stated in Definition \ref{contr}.
\end{proof}

\begin{lemma}\label{cyclsym} Assume, in addition to the  hypotheses of Proposition \ref{transa}, that for some $l>0$ we have
\[A=\widetilde{A} \otimes \textbf{M}_{l}(\textbf{k}),\quad B=\widetilde{B} \otimes \textbf{M}_{l}(\textbf{k}),\] where $\widetilde{A}$ is a commutative dg algebra,
and that the contraction between $A$ and $B$ is induced by a contraction between $\widetilde{A}$ and $\widetilde{B}$. Suppose further that  $x_{1},\ldots,x_{n} \in B$ are of even degree and $n$ is odd and greater than 1. Then the transferred curved $A_{\infty}$-structure on $B$  satisfies
\be{cycsym}\sum_{\sigma - \text{cyclic} }\text{Tr}(m_{n}^{B}(x_{\sigma(1)},\ldots, x_{\sigma(n)}))=0,\ee
where the sum is over all cyclic permutations and Tr denotes the matrix trace.
The same identity holds for the odd components of the morphism $I:B \rightarrow A^{\gamma}$  given by Proposition \ref{transa}.
\end{lemma}

\begin{proof}
By Proposition \ref{expform} it suffices to show that
\be{transum}
\sum_{\sigma - cyclic}\sum_{\; \tau \in \Upsilon _{n}^{dec}} (-1)^{\epsilon(\tau)}
\text{Tr}(m_{\tau}(x_{\sigma(1)},\ldots, x_{\sigma(n)}))
\ee
is identically 0.
Observe that for every $\tau$ we can write $m_{\tau}= p   m_{2}^{A} (N_{1},N_{2})$ for some maps
$N_{1}:B^{\otimes k}\rightarrow A$, $N_{2}:B^{\otimes n-k}\rightarrow A$ and some $k>0$. Using the definition of $\epsilon(\tau)$  one checks that the terms in
(\ref{transum}) corresponding to $p  m_{2}^{A} (N_{1},N_{2})$ and $p  m_{2}^{A} (N_{2},N_{1})$
appear with opposite signs. But it is easy to see that our parity assumptions imply
\begin{eqnarray}
&& \text{Tr}(m_{2}^{A} (N_{1}(x_{1},\ldots, x_{k}),N_{2}(x_{k+1},\ldots, x_{n})))
\nonumber \\ &&
= \ \text{Tr}(m_{2}^{A} (N_{2}(x_{k+1},\ldots, x_{n}),N_{1}(x_{1},\ldots, x_{k})))
\nonumber,
\end{eqnarray}
hence all terms in (\ref{transum}) cancel in pairs.
\end{proof}

\begin{lemma}\label{secondcw} Assume, in addition to the  hypotheses of Proposition \ref{transa}, that $A$ is commutative and extend the transferred curved $A_{\infty}$-structure on $B$ to $\textbf{M}_{l}(B)$. If we deform the latter structure by an element $\beta \in \textbf{M}_{l}(B_{1})$ to a a curved $A_{\infty}$-structure $\{m_{k}^{B,\beta}\}$ as in Proposition \ref{ainfdeform}, we have
\[\sum_{\sigma - \text{cyclic} }\text{Tr}(m_{n}^{B,\beta}(x_{\sigma(1)},\ldots, x_{\sigma(n)}))=0\]
for every odd $n>1$ and all $x_{1},\ldots,x_{n} \in \textbf{M}_{l}(B)$  of even degree.
\end{lemma}

\begin{proof}
Lemma \ref{cyclsym} easily implies  that the curved $A_{\infty}$-structure on  $\textbf{M}_{l}(B)$ satisfies (\ref{cycsym}). A computation similar to the proof of Lemma \ref{cyclsym} shows that the trace vanishes on the sums appearing in (\ref{deforma}) for every $k>0$, which proves the claim.
\end{proof}

\begin{example}\label{transaex} Let $X$ be a finite polyhedron. Given a  vector bundle $E$ on $X$, we can find a piece-wise smooth idempotent $P$ representing the $K$-theory class of $E$, as already noted in Sect. 2. Thus, we obtain a curved dg structure on the complex $\Omega^{\bullet}(X,\textbf{M}_{l}(\textbf{k}))$ of piece-wise smooth forms on $X$ as in Example \ref{ex3section1}. As in Sect. 2, we equip $\Omega^{\bullet}(X,\textbf{M}_{l}(\textbf{k}))$ with an $L_{2}$-norm and use the inclusion $W$ to transfer this norm to $C^{\bullet}(X,\textbf{M}_{l}(\textbf{k}))$. Assume that we have fixed fine enough subdivision $\widetilde{X}$ of $X$  so that the condition $\|\widehat{\gamma}\| \, \|H\|<1$ (where $H$ is the contraction from Theorem \ref{contraction} and $\widehat{\gamma}$ denotes commutator with $\gamma=
PdP$) is satisfied locally, i.e. on each simplex. Then  we can transfer the curved dg structure on $\Omega^{\bullet}(\widetilde{X},\textbf{M}_{l}(\textbf{k}))$ obtaining a curved $A_{\infty}$-structure on $C^{\bullet}(\widetilde{X},\textbf{M}_{l}(\textbf{k}))$ as in Proposition \ref{transa} (all series are locally convergent
as in Proposition \ref{exadef}).
\end{example}

We remark that the analogue of Proposition \ref{transa} for curved Lie dg structures holds. Indeed, in the absence of curvature, this is the main result of \cite{H2} (see also \cite{Be} for an alternative approach). The generalization to the curved case is done is as in the proof of Proposition \ref{transa}. As an application, one can transfer the curved dg Lie algebra structure associated to a principal bundle $P$ from Example \ref{cdglex} to a curved $L_{\infty}$-structure on the complex
of the $\mathfrak{g}$-valued cochains of a fine enough triangulation of the total space of $P$.

\medskip

We conclude this section by summarizing our results obtained so far as follows. Given a curved dg algebra (or curved Lie dg algebra) $A$  that is obtained as a deformation of a (Lie) dg algebra via a degree 1 element $\gamma$ and a special contraction from $A$ to a complex $B$, there are two procedures to construct a curved $A_{\infty}$(respectively $L_{\infty}$)-structure on $B$: One can first transfer the (Lie) dg algebra structure of $A$ to $B$ and then deform this transferred structure via the transferred element $\gamma '$ or one can directly transfer the curved (Lie) dg structure of $A$ to $B$. We shall see in the next section that one can obtain non-trivial characteristic classes from a curved $A_{\infty}$-structure obtained by the second procedure, while for curved $A_{\infty}$-structures obtained by the first procedure that appears to be impossible.

\section{A generalized Chern-Weil theory}

\subsection{Chern-Weil triples}

In this section, we develop an abstract algebraic version of the classical Chern-Weil algorithm
for producing characteristic classes of vector bundles.

\begin{definition}\label{cwtriple}
Let $(A, \{m_{k}\}_{k=0}^{\infty})$ be a curved $A_{\infty}$-algebra, $(C,d)$ a chain complex,
and $\Phi: A \rightarrow C$ a linear map of degree zero.
We call $(A,\Phi,C)$ a {\bf Chern-Weil triple} if for for every $x$ of even degree in $A$,
every odd $n>1$ and all even $x_{1},\ldots,x_{n}$ in $A$ the following  conditions hold.
\begin{eqnarray*}
\text{(i)} & \Phi (m_{1}(x)) \, = \, d(\Phi(x)) \\
\text{(ii)} & \sum_{\sigma \in \textbf{S}_{n}}\Phi(m_{n}(x_{\sigma (1)},\ldots, x_{\sigma(n)})) \, = \, 0 \,.
\end{eqnarray*}
\end{definition}

Given a Chern-Weil triple $(A,\Phi,C)$,
it is easy to check that $\Phi(m_{2}(m_{0},m_{0}))$ is a cocycle in $C$.
We shall now construct higher cocycles in $C$ by introducing higher powers of the curvature $m_{0}$.


In order to formulate our next result, we
recall that there is a one-to-one correspondence between rooted planar trees and bracket expressions,
for example:
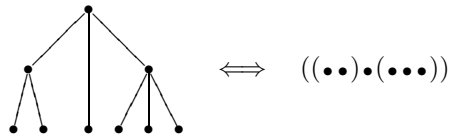
\[
\xygraph{
!{<0cm,0cm>;<1cm,0cm>:<0cm,0.8cm>::}
!{(0,1)}*{\Blt}="1"
!{(-0.8,0)}*{\Blt}="2"
!{(0.8,0)}*{\Blt}="3"
!{(-1,-1)}*{\Blt}="4"
!{(-0.6,-1)}*{\Blt}="5"
!{(0,-1)}*{\Blt}="6"
!{(0.4,-1)}*{\Blt}="7"
!{(0.8,-1)}*{\Blt}="8"
!{(1.2,-1)}*{\Blt}="9"
"1"-"2" "1"-"3" "1"-"6" "2"-"4" "2"-"5" "3"-"7" "3"-"8" "3"-"9"
}
\quad \Longleftrightarrow \quad ((\blt \blt) \blt(\blt \blt \blt))
\]
Using this correspondence, given a curved $A_{\infty}$-algebra $A$, we define an element $\mathcal{M}_{\TRE}$ in $A$ for every rooted planar tree $\TRE$ by assigning to each vertex of valency $k$ the operation $m_{k}$ in $A$.
Thus in the example above we have:
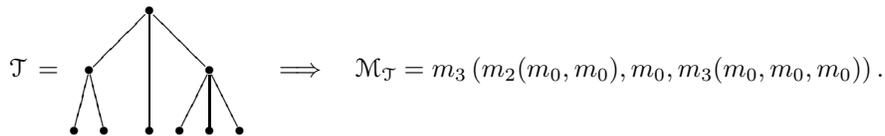
\[
\TRE \, = \hspace{1pt}
\xygraph{
!{<0cm,0cm>;<1cm,0cm>:<0cm,0.8cm>::}
!{(0,1)}*{\Blt}="1"
!{(-0.8,0)}*{\Blt}="2"
!{(0.8,0)}*{\Blt}="3"
!{(-1,-1)}*{\Blt}="4"
!{(-0.6,-1)}*{\Blt}="5"
!{(0,-1)}*{\Blt}="6"
!{(0.4,-1)}*{\Blt}="7"
!{(0.8,-1)}*{\Blt}="8"
!{(1.2,-1)}*{\Blt}="9"
"1"-"2" "1"-"3" "1"-"6" "2"-"4" "2"-"5" "3"-"7" "3"-"8" "3"-"9"
}
\quad \Longrightarrow \quad
\mathcal{M}_{\TRE}=m_{3}\left(m_{2}(m_{0},m_{0}),m_{0},m_{3}(m_{0},m_{0},m_{0})\right).
\]
We call the degree of the element $\mathcal{M}_{\TRE}$ in $A$ (which is independent of $A$)
the {\bf degree} of the tree $\TRE$.
We say that a rooted planar tree is {\bf completely even}
if all of its vertices have even valency and denote the set of all completely even trees of degree $n$
by
$
\CET(n) \,
$.
Given a rooted tree, consider for each tail the number of edges in the shortest path joining the tail and the root. We call the largest among these numbers the {\bf depth} of the rooted tree.

\begin{remark}\label{newrem1}
It is instructive
to note that every completely even tree  can be  presented~as
\[
\TRE \, = \hspace{1pt}
\xygraph{
!{<0cm,0cm>;<1.5cm,0cm>:<0cm,1cm>::}
!{(0,0.6)}*{\Blt}="1"
!{(-0.5,-0.1)}*{\Blt}="2"
!{(0,-0.1)}*{\ldots}="3"
!{(0.5,-0.1)}*{\Blt}="4"
!{(-0.7,-0.6)}*{.}="5"
!{(-0.3,-0.6)}*{.}="6"
!{(0.3,-0.6)}*{.}="7"
!{(0.7,-0.6)}*{.}="8"
!{(-0.48,-0.48)}*{\TRE_1}="9"
!{(0.53,-0.48)}*{\TRE_n}="10"
"1"-"2" "1"-"4" "2"-"5" "2"-"6" "4"-"7" "4"-"8"
}
\quad \Longrightarrow \quad
\mathcal{M}_{\TRE} \, = \, m_n \bigl(\mathcal{M}_{\TRE_1},\dots,\mathcal{M}_{\TRE_n}\bigr) \,.
\]
Hence, degree of $\mathcal{M}_{\TRE}$ $=$ $2$ $-$ $n$ $+$ degree of $\mathcal{M}_{\TRE_1}$ $+$ $\dots$ $+$
degree of $\mathcal{M}_{\TRE_n}$, in other words,
degree of $\mathcal{M}_{\TRE}$ $-$ $1$ $=$ $1$ $+$ (degree of $\mathcal{M}_{\TRE_1}-1$) $+$ $\cdots$ $+$
(degree of $\mathcal{M}_{\TRE_n}-1$).
It follows then by induction on the number of  vertices of $\TRE$ that
\[
\text{degree of } \TRE \, = \, 1 + \text{number of vertices of } \TRE \,.
\]
\end{remark}

\begin{theorem}\label{defchar}
Let $(A,\Phi,C)$ be a Chern-Weil triple. Then
\be{cocycles}
\mathcal{C}_{n} \, = \, \sum_{\TRE \, \in \, \CET(2n)}\Phi\left({\mathcal{M}_{\TRE}}\right)
\ee
is a cocycle in $C^{2n}$ for every $n>0$.
\end{theorem}

\begin{proof}
According to property (i) in Definition \ref{cwtriple}, it suffices to show that
\[\sum_{\TRE \, \in \, \CET(2n)}m_{1}({\mathcal{M}_{\TRE}})
\]
is in the kernel of $\Phi$. This will be done using the following formula
valid for every completely even tree $\TRE$
(in fact, for every tree with no vertex of valency $1$):
\be{formulam1}
m_{1}(\mathcal{M}_{\TRE}) \, = \,
\sum_{\opr \, \in \, B(\TRE)}(-1)^{\varepsilon(\TRE,\OPR{\TRE})} \, \mathcal{M}_{\OPR{\TRE}}\,,
\ee
where $B(\TRE)$ is a set consisting of two types of operations $\TRE \, \mathop{\longmapsto}\limits^{\opr} \, \OPR{\TRE}$
 defined as follows:

(a) {\it Attaching} an edge to a vertex of $\TRE$ which is not a
tail.
The new vertex of the added edge then becomes a new
tail
for the resulting tree.
This operation can  schematically be represented as follows:
\be{op1}
\xygraph{
!{<0cm,0cm>;<0.8cm,0cm>:<0cm,2.4cm>::}
!{(0,0.5)}*{\Blt}="1"
!{(-0.45,0.35)}*{.}="2"
!{(0,0.05)}*{\Blt}="3"
!{(0.4,0.09)}*{m_k}
!{(0.45,0.35)}*{.}="4"
!{(-1.5,-0.25)}*{\Blt}="5"
!{(0.13,-0.08)}*{\text{\scriptsize $j$}}
!{(-0.63,-0.23)}*{.\,.\,.}
!{(0.63,-0.23)}*{.\,.\,.}
!{(0,-0.25)}*{\Blt}="6"
!{(1.5,-0.25)}*{\Blt}="7"
!{(-3,-0.5)}*{.}="8"
!{(3,-0.5)}*{.}="9"
!{(-1.7,-0.5)}*{.}="10"
!{(-1.3,-0.5)}*{.}="11"
!{(-0.2,-0.5)}*{.}="13"
!{(0.2,-0.5)}*{.}="14"
!{(1.3,-0.5)}*{.}="15"
!{(1.7,-0.5)}*{.}="16"
!{(0,0.35)}*{.}="3a"
!{(0,0.2)}*{.}="3b"
!{(0,0.23)}*{.}!{(0,0.26)}*{.}!{(0,0.29)}*{.}!{(0,0.32)}*{.}
!{(-0.9,0.2)}*{.}!{(-0.81,0.23)}*{.}!{(-0.72,0.26)}*{.}!{(-0.63,0.29)}*{.}!{(-0.54,0.32)}*{.}
!{(-0.99,0.17)}*{.}!{(-1.08,0.14)}*{.}!{(-1.17,0.11)}*{.}!{(-1.26,0.08)}*{.}!{(-1.35,0.05)}*{.}!{(-1.44,0.02)}*{.}!{(-1.53,-0.01)}*{.}!{(-1.62,-0.04)}*{.}!{(-1.71,-0.07)}*{.}!{(-1.8,-0.1)}*{.}!{(-1.89,-0.13)}*{.}!{(-1.98,-0.16)}*{.}!{(-2.07,-0.19)}*{.}!{(-2.16,-0.22)}*{.}!{(-2.25,-0.25)}*{.}!{(-2.34,-0.28)}*{.}!{(-2.43,-0.31)}*{.}!{(-2.52,-0.34)}*{.}!{(-2.61,-0.37)}*{.}!{(-2.7,-0.4)}*{.}!{(-2.79,-0.43)}*{.}!{(-2.895,-0.465)}*{.}
!{(0.9,0.2)}*{.}!{(0.81,0.23)}*{.}!{(0.72,0.26)}*{.}!{(0.63,0.29)}*{.}!{(0.54,0.32)}*{.}
!{(0.99,0.17)}*{.}!{(1.08,0.14)}*{.}!{(1.17,0.11)}*{.}!{(1.26,0.08)}*{.}!{(1.35,0.05)}*{.}!{(1.44,0.02)}*{.}!{(1.53,-0.01)}*{.}!{(1.62,-0.04)}*{.}!{(1.71,-0.07)}*{.}!{(1.8,-0.1)}*{.}!{(1.89,-0.13)}*{.}!{(1.98,-0.16)}*{.}!{(2.07,-0.19)}*{.}!{(2.16,-0.22)}*{.}!{(2.25,-0.25)}*{.}!{(2.34,-0.28)}*{.}!{(2.43,-0.31)}*{.}!{(2.52,-0.34)}*{.}!{(2.61,-0.37)}*{.}!{(2.7,-0.4)}*{.}!{(2.79,-0.43)}*{.}!{(2.895,-0.465)}*{.}
!{(-0.75,-0.5)}*{.\,.\,.}
!{(0.75,-0.5)}*{.\,.\,.}
!{(-1.5,-0.4)}*{.}!{(-1.5,-0.44)}*{.}!{(-1.5,-0.48)}*{.}
!{(1.5,-0.4)}*{.}!{(1.5,-0.44)}*{.}!{(1.5,-0.48)}*{.}
!{(0,-0.4)}*{.}!{(0,-0.44)}*{.}!{(0,-0.48)}*{.}
"1"-"2" "1"-"3a" "1"-"4"
"3"-"5" "3b"-"3" "3"-"6" "3"-"7"
"5"-"10" "5"-"11" "7"-"15" "7"-"16" "6"-"13" "6"-"14"
}
\hspace{10pt}
\mathop{\longmapsto}\limits^{\opr}
\hspace{10pt}
\xygraph{
!{<0cm,0cm>;<0.8cm,0cm>:<0cm,2.4cm>::}
!{(0,0.5)}*{\Blt}="1"
!{(-0.45,0.35)}*{.}="2"
!{(0,0.05)}*{\Blt}="3"
!{(0.4,0.09)}*{\hspace{8.5pt} m_{k+1}}
!{(0.45,0.35)}*{.}="4"
!{(-1.5,-0.25)}*{\Blt}="5"
!{(-0.3,-0.13)}*{\text{\scriptsize $j$}}
!{(0.35,-0.13)}*{\text{\scriptsize $j$+1}}
!{(-0.8,-0.23)}*{.\,.\,.}
!{(0.63,-0.23)}*{.\,.\,.}
!{(0,-0.25)}*{\Blt}="6"
!{(1.5,-0.25)}*{\Blt}="7"
!{(-3,-0.5)}*{.}="8"
!{(3,-0.5)}*{.}="9"
!{(-1.7,-0.5)}*{.}="10"
!{(-1.3,-0.5)}*{.}="11"
!{(-0.5,-0.5)}*{\Blt}="12"
!{(-0.2,-0.5)}*{.}="13"
!{(0.2,-0.5)}*{.}="14"
!{(1.3,-0.5)}*{.}="15"
!{(1.7,-0.5)}*{.}="16"
!{(0,0.35)}*{.}="3a"
!{(0,0.2)}*{.}="3b"
!{(0,0.23)}*{.}!{(0,0.26)}*{.}!{(0,0.29)}*{.}!{(0,0.32)}*{.}
!{(-0.9,0.2)}*{.}!{(-0.81,0.23)}*{.}!{(-0.72,0.26)}*{.}!{(-0.63,0.29)}*{.}!{(-0.54,0.32)}*{.}
!{(-0.99,0.17)}*{.}!{(-1.08,0.14)}*{.}!{(-1.17,0.11)}*{.}!{(-1.26,0.08)}*{.}!{(-1.35,0.05)}*{.}!{(-1.44,0.02)}*{.}!{(-1.53,-0.01)}*{.}!{(-1.62,-0.04)}*{.}!{(-1.71,-0.07)}*{.}!{(-1.8,-0.1)}*{.}!{(-1.89,-0.13)}*{.}!{(-1.98,-0.16)}*{.}!{(-2.07,-0.19)}*{.}!{(-2.16,-0.22)}*{.}!{(-2.25,-0.25)}*{.}!{(-2.34,-0.28)}*{.}!{(-2.43,-0.31)}*{.}!{(-2.52,-0.34)}*{.}!{(-2.61,-0.37)}*{.}!{(-2.7,-0.4)}*{.}!{(-2.79,-0.43)}*{.}!{(-2.895,-0.465)}*{.}
!{(0.9,0.2)}*{.}!{(0.81,0.23)}*{.}!{(0.72,0.26)}*{.}!{(0.63,0.29)}*{.}!{(0.54,0.32)}*{.}
!{(0.99,0.17)}*{.}!{(1.08,0.14)}*{.}!{(1.17,0.11)}*{.}!{(1.26,0.08)}*{.}!{(1.35,0.05)}*{.}!{(1.44,0.02)}*{.}!{(1.53,-0.01)}*{.}!{(1.62,-0.04)}*{.}!{(1.71,-0.07)}*{.}!{(1.8,-0.1)}*{.}!{(1.89,-0.13)}*{.}!{(1.98,-0.16)}*{.}!{(2.07,-0.19)}*{.}!{(2.16,-0.22)}*{.}!{(2.25,-0.25)}*{.}!{(2.34,-0.28)}*{.}!{(2.43,-0.31)}*{.}!{(2.52,-0.34)}*{.}!{(2.61,-0.37)}*{.}!{(2.7,-0.4)}*{.}!{(2.79,-0.43)}*{.}!{(2.895,-0.465)}*{.}
!{(-0.9,-0.5)}*{.\hspace{1pt}.\hspace{1pt}.}
!{(0.75,-0.5)}*{.\,.\,.}
!{(-1.5,-0.4)}*{.}!{(-1.5,-0.44)}*{.}!{(-1.5,-0.48)}*{.}
!{(1.5,-0.4)}*{.}!{(1.5,-0.44)}*{.}!{(1.5,-0.48)}*{.}
!{(0,-0.4)}*{.}!{(0,-0.44)}*{.}!{(0,-0.48)}*{.}
"1"-"2" "1"-"3a" "1"-"4"
"3"-"5" "3b"-"3" "3"-"6" "3"-"7"
"5"-"10" "5"-"11" "7"-"15" "7"-"16" "3"-"12" "6"-"13" "6"-"14"
}
\ee
Note that if this operation is performed on a vertex of valency $k$ there are
$k+1$ possibilities for attachment. We assign to the tree $\OPR{\TRE}$
so obtained the sign $(-1)^{\varepsilon(\TRE,\OPR{\TRE} )}=(-1)^{j+1}$,
where $j=0,\ldots,k$ represents the position of the attached edge moving from left to right.

(b) {\it Grafting} an edge at a vertex which is not a
tail.
We can represent this operation schematically as follows:
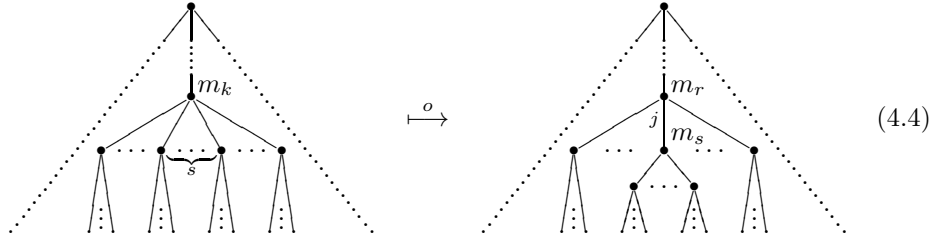
\be{op2}
\xygraph{
!{<0cm,0cm>;<0.8cm,0cm>:<0cm,3cm>::}
!{(0,0.5)}*{\Blt}="1"
!{(-0.45,0.35)}*{.}="2"
!{(0,0.1)}*{\Blt}="3"
!{(0.4,0.14)}*{m_k}
!{(0.45,0.35)}*{.}="4"
!{(-1.5,-0.14)}*{\Blt}="5"
!{(-0.5,-0.14)}*{\Blt}="6"
!{(0.5,-0.14)}*{\Blt}="6a"
!{(1.5,-0.14)}*{\Blt}="7"
!{(-1,-0.14)}*{.\,.\,.}
!{(1,-0.14)}*{.\,.\,.}
!{(0,-0.14)}*{.\,.\,.}
!{(0,-0.18)}*{\text{\tiny $\underbrace{\hspace{19pt}}$}}
!{(0,-0.23)}*{\text{\scriptsize $s$}}
!{(-3,-0.5)}*{.}="8"
!{(3,-0.5)}*{.}="9"
!{(-1.7,-0.5)}*{.}="10"
!{(-1.3,-0.5)}*{.}="11"
!{(-0.7,-0.5)}*{.}="13"
!{(-0.3,-0.5)}*{.}="14"
!{(0.7,-0.5)}*{.}="13a"
!{(0.3,-0.5)}*{.}="14a"
!{(1.3,-0.5)}*{.}="15"
!{(1.7,-0.5)}*{.}="16"
!{(0,0.35)}*{.}="3a"
!{(0,0.2)}*{.}="3b"
!{(0,0.23)}*{.}!{(0,0.26)}*{.}!{(0,0.29)}*{.}!{(0,0.32)}*{.}
!{(-0.9,0.2)}*{.}!{(-0.81,0.23)}*{.}!{(-0.72,0.26)}*{.}!{(-0.63,0.29)}*{.}!{(-0.54,0.32)}*{.}
!{(-0.99,0.17)}*{.}!{(-1.08,0.14)}*{.}!{(-1.17,0.11)}*{.}!{(-1.26,0.08)}*{.}!{(-1.35,0.05)}*{.}!{(-1.44,0.02)}*{.}!{(-1.53,-0.01)}*{.}!{(-1.62,-0.04)}*{.}!{(-1.71,-0.07)}*{.}!{(-1.8,-0.1)}*{.}!{(-1.89,-0.13)}*{.}!{(-1.98,-0.16)}*{.}!{(-2.07,-0.19)}*{.}!{(-2.16,-0.22)}*{.}!{(-2.25,-0.25)}*{.}!{(-2.34,-0.28)}*{.}!{(-2.43,-0.31)}*{.}!{(-2.52,-0.34)}*{.}!{(-2.61,-0.37)}*{.}!{(-2.7,-0.4)}*{.}!{(-2.79,-0.43)}*{.}!{(-2.895,-0.465)}*{.}
!{(0.9,0.2)}*{.}!{(0.81,0.23)}*{.}!{(0.72,0.26)}*{.}!{(0.63,0.29)}*{.}!{(0.54,0.32)}*{.}
!{(0.99,0.17)}*{.}!{(1.08,0.14)}*{.}!{(1.17,0.11)}*{.}!{(1.26,0.08)}*{.}!{(1.35,0.05)}*{.}!{(1.44,0.02)}*{.}!{(1.53,-0.01)}*{.}!{(1.62,-0.04)}*{.}!{(1.71,-0.07)}*{.}!{(1.8,-0.1)}*{.}!{(1.89,-0.13)}*{.}!{(1.98,-0.16)}*{.}!{(2.07,-0.19)}*{.}!{(2.16,-0.22)}*{.}!{(2.25,-0.25)}*{.}!{(2.34,-0.28)}*{.}!{(2.43,-0.31)}*{.}!{(2.52,-0.34)}*{.}!{(2.61,-0.37)}*{.}!{(2.7,-0.4)}*{.}!{(2.79,-0.43)}*{.}!{(2.895,-0.465)}*{.}
!{(-1.5,-0.4)}*{.}!{(-1.5,-0.44)}*{.}!{(-1.5,-0.48)}*{.}
!{(1.5,-0.4)}*{.}!{(1.5,-0.44)}*{.}!{(1.5,-0.48)}*{.}
!{(-0.5,-0.4)}*{.}!{(-0.5,-0.44)}*{.}!{(-0.5,-0.48)}*{.}
!{(0.5,-0.4)}*{.}!{(0.5,-0.44)}*{.}!{(0.5,-0.48)}*{.}
"1"-"2" "1"-"3a" "1"-"4"
"3"-"5" "3b"-"3" "3"-"6" "3"-"6a" "3"-"7"
"5"-"10" "5"-"11" "7"-"15" "7"-"16" "6"-"13" "6"-"14" "6a"-"13a" "6a"-"14a"
}
\hspace{10pt}
\mathop{\longmapsto}\limits^{\opr}
\hspace{10pt}
\xygraph{
!{<0cm,0cm>;<0.8cm,0cm>:<0cm,3cm>::}
!{(0,0.5)}*{\Blt}="1"
!{(-0.45,0.35)}*{.}="2"
!{(0,0.1)}*{\Blt}="3"
!{(0.4,0.14)}*{m_r}
!{(0.45,0.35)}*{.}="4"
!{(-1.5,-0.14)}*{\Blt}="5"
!{(0,-0.14)}*{\Blt}="6z"
!{(0.4,-0.075)}*{m_s}
!{(-0.16,0)}*{\text{\scriptsize $j$}}
!{(-0.5,-0.3)}*{\Blt}="6"
!{(0.5,-0.3)}*{\Blt}="6a"
!{(1.5,-0.14)}*{\Blt}="7"
!{(-0.75,-0.14)}*{.\,.\,.}
!{(0.75,-0.14)}*{.\,.\,.}
!{(0,-0.3)}*{.\,.\,.}
!{(-3,-0.5)}*{.}="8"
!{(3,-0.5)}*{.}="9"
!{(-1.7,-0.5)}*{.}="10"
!{(-1.3,-0.5)}*{.}="11"
!{(-0.7,-0.5)}*{.}="13"
!{(-0.3,-0.5)}*{.}="14"
!{(0.7,-0.5)}*{.}="13a"
!{(0.3,-0.5)}*{.}="14a"
!{(1.3,-0.5)}*{.}="15"
!{(1.7,-0.5)}*{.}="16"
!{(0,0.35)}*{.}="3a"
!{(0,0.2)}*{.}="3b"
!{(0,0.23)}*{.}!{(0,0.26)}*{.}!{(0,0.29)}*{.}!{(0,0.32)}*{.}
!{(-0.9,0.2)}*{.}!{(-0.81,0.23)}*{.}!{(-0.72,0.26)}*{.}!{(-0.63,0.29)}*{.}!{(-0.54,0.32)}*{.}
!{(-0.99,0.17)}*{.}!{(-1.08,0.14)}*{.}!{(-1.17,0.11)}*{.}!{(-1.26,0.08)}*{.}!{(-1.35,0.05)}*{.}!{(-1.44,0.02)}*{.}!{(-1.53,-0.01)}*{.}!{(-1.62,-0.04)}*{.}!{(-1.71,-0.07)}*{.}!{(-1.8,-0.1)}*{.}!{(-1.89,-0.13)}*{.}!{(-1.98,-0.16)}*{.}!{(-2.07,-0.19)}*{.}!{(-2.16,-0.22)}*{.}!{(-2.25,-0.25)}*{.}!{(-2.34,-0.28)}*{.}!{(-2.43,-0.31)}*{.}!{(-2.52,-0.34)}*{.}!{(-2.61,-0.37)}*{.}!{(-2.7,-0.4)}*{.}!{(-2.79,-0.43)}*{.}!{(-2.895,-0.465)}*{.}
!{(0.9,0.2)}*{.}!{(0.81,0.23)}*{.}!{(0.72,0.26)}*{.}!{(0.63,0.29)}*{.}!{(0.54,0.32)}*{.}
!{(0.99,0.17)}*{.}!{(1.08,0.14)}*{.}!{(1.17,0.11)}*{.}!{(1.26,0.08)}*{.}!{(1.35,0.05)}*{.}!{(1.44,0.02)}*{.}!{(1.53,-0.01)}*{.}!{(1.62,-0.04)}*{.}!{(1.71,-0.07)}*{.}!{(1.8,-0.1)}*{.}!{(1.89,-0.13)}*{.}!{(1.98,-0.16)}*{.}!{(2.07,-0.19)}*{.}!{(2.16,-0.22)}*{.}!{(2.25,-0.25)}*{.}!{(2.34,-0.28)}*{.}!{(2.43,-0.31)}*{.}!{(2.52,-0.34)}*{.}!{(2.61,-0.37)}*{.}!{(2.7,-0.4)}*{.}!{(2.79,-0.43)}*{.}!{(2.895,-0.465)}*{.}
!{(-1.5,-0.4)}*{.}!{(-1.5,-0.44)}*{.}!{(-1.5,-0.48)}*{.}
!{(1.5,-0.4)}*{.}!{(1.5,-0.44)}*{.}!{(1.5,-0.48)}*{.}
!{(-0.5,-0.4)}*{.}!{(-0.5,-0.44)}*{.}!{(-0.5,-0.48)}*{.}
!{(0.5,-0.4)}*{.}!{(0.5,-0.44)}*{.}!{(0.5,-0.48)}*{.}
"1"-"2" "1"-"3a" "1"-"4"
"3"-"5" "3b"-"3" "3"-"6z" "6z"-"6" "6z"-"6a" "3"-"7"
"5"-"10" "5"-"11" "7"-"15" "7"-"16" "6"-"13" "6"-"14" "6a"-"13a" "6a"-"14a"
}
\ee
\[
k \, = \, r + s - 1 \,,\quad 1 \, < \, s \hspace{1pt},\hspace{1pt} r \, < \,k
\]
In detail, we choose $s>1$ adjacent edges coming out of a vertex of valency $k$ and attach them to a new vertex
of valency $s$ together with the subtrees attached to them, then join the vertex of valency $s$
with the original vertex which now becomes of valency $r=k-s+1$.
In this case we have
\be{signstrees}
(-1)^{\varepsilon(\TRE,\OPR{\TRE})} \, = \, (-1)^{rs+(j+1)(s+1)} \,,
\ee
where $j=0,\dots,r-1$ represents the position of the group of $s$ adjacent edges.

We remark that one can treat the operations of type (a) as ``generalized'' operations of type (b)
with $s=0$ so that Eq. (\ref{signstrees}) remains true in the case of operations of type (a) as well.

We verify Eq. (\ref{formulam1}) using induction on the depth
of the tree $\TRE$. If this depth is 1 the identity (\ref{curvedAinfinity}) gives
\begin{multline*}
m_{1}\bigl(m_{k}(m_{0},\ldots,m_{0})\bigr) \, = \,
\sum_{j \, = \, 0}^{k} \,
(-1)^{j+1}
\, m_{k+1}(m_{0},\ldots,m_{0})+\\
+\sum_{\mathop{}\limits^{r+s-1 \, = \, k}_{s \, > \, 1}} \, \sum_{j \, = \, 0}^{r-1} \,
(-1)^{rs+(j+1)(s+1)}
\, m_{r}\bigl(m_{0},\ldots,m_{0},m_{s}(m_{0},\ldots,m_{0}),m_{0},\ldots,m_{0}\bigr)+\\
+ \sum_{j \, = \, 0}^{k-1} \,
(-1)^k
\, m_{k}(m_{0},\ldots,m_{0},m_{1}(m_{0}),m_{0},\ldots,m_{0}).
\end{multline*}
Note that the last sum vanishes due to the identity $m_{1}(m_{0})=0$.
Hence, we obtain formula (\ref{formulam1}) in the case of tree depth $1$ since the sum over the operations of type (a) corresponds to the sum in the first line and the sum over the operations of type (b) corresponds to the sum in the second line above.

If $\TRE$ has depth greater than $1$ we can write,
as in Remark~\ref{newrem1},
$\mathcal{M}_{\TRE}=m_{k}\bigl(\mathcal{M}_{\TRE_{1}},$ $\dots,$ $\mathcal{M}_{\TRE_{k}}\bigr)$
for some trees
$\TRE_{1},\ldots,\TRE_{k}$ whose depths are smaller than the depth of $\TRE$.
Using (\ref{curvedAinfinity}) again we obtain
\begin{multline*}
m_{1}\bigl(\mathcal{M}_{\TRE}\bigr) \, = \,
\sum_{j \, = \, 0}^{k} \,
(-1)^{j+1}
\, m_{k+1}\bigl(\mathcal{M}_{\TRE_{1}},\ldots,\mathcal{M}_{\TRE_{j}},
m_{0},\mathcal{M}_{\TRE_{j+1}},\ldots,\mathcal{M}_{\TRE_{k}}\bigr)+\\
+\sum_{\mathop{}\limits^{r+s-1 \, = \, k}_{s \, > \, 1}} \, \sum_{j \, = \, 0}^{r-1} \,
(-1)^{rs+(j+1)(s+1)}
\, m_{r}\bigl(\mathcal{M}_{\TRE_{1}},\ldots,m_{s}(\mathcal{M}_{\TRE_{j}},\ldots,\mathcal{M}_{\TRE_{j+s}})
,\ldots,\mathcal{M}_{\TRE_{r}}\bigr)+\\
+ \sum_{j \, = \, 0}^{k-1} \,
(-1)^k
\, m_{k}(\mathcal{M}_{\TRE_{1}},\ldots,m_{1}\bigl(\mathcal{M}_{\TRE_{j}}\bigr),\ldots,\mathcal{M}_{\TRE_{k}}\bigr).
\end{multline*}
Applying the inductive hypothesis to $m_{1}(\mathcal{M}_{\TRE_{j}})$ in the last sum,
we obtain (\ref{formulam1}), as we did above for the case of depth $1$.

It remains to show that
\be{treesum}
\sum_{\TRE \, \in \, \CET(2n)} \,
m_1 \bigl(\mathcal{M}_{\TRE}) \, = \,
\sum_{\TRE \, \in \, \CET(2n)} \ \sum_{\opr \, \in \, B(\TRE)} \, (-1)^{\varepsilon(\TRE,\OPR{\TRE})} \,
\mathcal{M}_{\OPR{\TRE}}
\ee
is in the kernel of $\Phi$.

Observe that all trees $\OPR{\TRE}$ in (\ref{treesum}) contain exactly one vertex of odd valency.
Indeed, this is the vertex where an operation either of type (a) or type (b) has been performed.
More precisely, in the type (b) case,
the operation is
applied to a vertex of even valency $k$
and we obtain two new vertices of valencies $r$ and $s$, respectively,
so that
exactly one of these two numbers is odd
due to the relation $r+s=k+1$.
When $s$ is even we will say that the operation is performed ``{\em below}'' the vertex of odd valency
in $\OPR{\TRE}$, otherwise we will say that the operation is performed ``{\em above}'' the vertex.
If the operation performed is of type (a) we will treat it as a ``generalized'' operation of type (b)
with $s=0$ and will say that is performed ``below'' the vertex.

We separate the sum in (\ref{treesum}) into two subsums: (1) terms $\mathcal{M}_{\OPR{\TRE}}$
for which the vertex of odd valency is not the root of
$\OPR{\TRE}$ and (2) terms $\mathcal{M}_{\OPR{\TRE}}$ for which this vertex is the root of $\OPR{\TRE}$.

We claim that the subsum (1) is identically zero.
Indeed, to each tree $\OPR{\TRE}$ whose odd vertex is not the root there correspond even number terms
in the subsum (1) which appear with alternating signs.
To show this, observe that such a tree $\OPR{\TRE}$ can be obtained from completely even trees $\TRE$
by performing one operation ``above'' the odd vertex and an odd number
(which  equals to the valency of this vertex) of  operations performed ``below''.
Now using (\ref{signstrees}) we see that the term that corresponds to the operation performed ``above''
appears with positive sign and the remaining odd number of terms appear with alternating signs,
$(-1)^{rs+(j+1)(s+1)}$ $=$ $(-1)^{j+1}$ ($j=0,\dots,r-1$),
beginning with minus.
All these terms cancel and the claim is proved.

Now consider the subsum (2).
To each tree $\OPR{\TRE}$ whose odd vertex is the root there correspond odd terms in this subsum.
Indeed, such a tree $\OPR{\TRE}$ can be obtained from $\TRE$ by performing odd number of operations ``below'' only.
Using (\ref{signstrees}) again we conclude that the signs of these terms alternate beginning with minus sign.
Thus, after cancelation only one term corresponding to every  $\OPR{\TRE}$ remains.
As a result, we obtain
\be{REC1}
\sum_{\TRE \, \in \, \CET(2n)} \,
m_1 \bigl(\mathcal{M}_{\TRE}) \, = \,
- \smtwo{k \text{ -- odd}}{1 < k \leqslant n} \
\smtwo{\TRE_j \in \CET(2n_j)}{n_1+\cdots+n_k = n}
m_k \bigl(\mathcal{M}_{\TRE_1},\dots,\mathcal{M}_{\TRE_k}\bigr) \,.
\ee
Now Property (ii)
in Definition \ref{cwtriple} implies that the subsum (2) is in the kernel of $\Phi$ and the theorem is proved.
\end{proof}

\begin{remark}
The action of $m_{1}$ on the trees considered above essentially coincides with that of the differential in
Kontsevich's graph cohomology complex \cite{K}. Namely, the result of the application of $m_{1}$ is a sum
(with appropriate signs) of trees that are obtained from the original tree by expanding a vertex to an edge.
\end{remark}

\begin{definition}
We call the cohomology classes defined by $\frac{1}{n!c_{n-1}}\cdot\mathcal{C}_{n},n=1,2,\ldots$
the components of the {\bf  Chern character} of the Chern-Weil triple
$(A,\hspace{-0.5pt}\Phi,\hspace{-0.5pt}C)$
and denote them by $\textbf{Ch}^{\bullet}(A,\hspace{-0.5pt}\Phi,\hspace{-0.5pt}C)$.
\end{definition}

In the definition above we divide by ${c_{n-1}}$, the number of all {\it binary} trees appearing in the sum (\ref{cocycles}), in order to obtain the usual definition of Chern character when $A$ is associative.

\subsection{Pre-Chern elements}\label{PreCh}

Equation (\ref{cocycles}) suggests that we introduce
\[
\CCW_{n} \, := \,
\sum_{\TRE \, \in \, \CET(2n)} \
\mathcal{M}_{\TRE}
\]
(Eq. (\ref{cocycles}) then reads $\mathcal{C}_{n} = \Phi\bigl(\CCW_{n}\bigr)$),
which we call \textbf{pre-Chern elements}.
From the recursion in Remark \ref{newrem1} we obtain
\be{REC2}
\CCW_n \, = \,
\smtwo{k \text{ -- even}}{2 \leqslant k \leqslant n} \
\sum_{n_1+\cdots+n_k \, = \, n}
\, m_k \bigl(\CCW_{n_1},\dots,\CCW_{n_k}\bigr) \,.
\ee
On the other hand, Eq.~(\ref{REC1}) implies that
\be{EQ1}
m_1 \bigl(\CCW_n\bigr) \, = \,
- \smtwo{k \text{ -- odd}}{1 < k \leqslant n} \
\sum_{n_1+\cdots+n_k \, = \, n}
\, m_k \bigl(\CCW_{n_1},\dots,\CCW_{n_k}\bigr) \,,\quad \text{or}
\ee
\be{EQ1-1}
\smtwo{k \text{ -- odd}}{1 \leqslant k \leqslant n} \
\sum_{n_1+\cdots+n_k \, = \, n}
\, m_k \bigl(\CCW_{n_1},\dots,\CCW_{n_k}\bigr) \, = \, 0\,.
\ee
Thus, we see that Eq.~(\ref{EQ1}), which can be easily verified using Eq.~(\ref{REC2}) and Eq.~(\ref{Ainfide}), is essentially what is needed to prove Theorem~\ref{defchar}.

It is instructive to give an independent proof of Eq.~(\ref{EQ1-1}).
To this end it is convenient to introduce the formal sum
\[
\mbf{\CCW} \, = \, \mathop{\sum}\limits_{n \, = \, 1}^{\infty} \CCW_n \,,
\]
which is nothing but the formal generating series of all $\CCW_n$:
the degree $2n$ part of $\CCW$ is $\CCW_n$.
Then Eq.~(\ref{REC2}) becomes
\be{REC3}
\mbf{\CCW} \, = \, \mathop{\sum}\limits_{k \, = \, 0}^{\infty} m_{2k} \Bigl(\mbf{\CCW}^{\otimes \, 2k}\Bigr) \, ,
\ee
while Eq.~(\ref{EQ1-1}) reads
\be{EQ1-2}
\mathop{\sum}\limits_{k \, = \, 1}^{\infty} m_{2k-1} \Bigl(\mbf{\CCW}^{\otimes \, (2k-1)}\Bigr) \, = \, 0 \,.
\ee
Now, to prove (\ref{EQ1-2}) we use the $A_{\infty}$-identity (\ref{Ainfide}) to find that
\be{tmpeq-x1}
\sum_{r+s+t \, = \, n}
(-1)^{r+st} \, m_{r+t+1}
\Bigl(\mbf{\CCW}^{\otimes{r}}\otimes m_{s} \bigl( \mbf{\CCW}^{\otimes{s}}\bigr) \otimes \mbf{\CCW}^{\otimes{t}}\Bigr)
\, = \, 0 \,.
\ee
Setting
$\mbf{\CCW'} := \mathop{\sum}\limits_{k \, = \, 1}^{\infty} m_{2k-1} \Bigl(\mbf{\CCW}^{\otimes \, (2k-1)}\Bigr)$
and summing (\ref{tmpeq-x1}) over {\it even} $n$,
we obtain
\begin{align}\label{calc2}
0 \, & = \,
\mathop{\sum}\limits_{N \, = \, 0}^{\infty} \
\mathop{\sum}\limits_{r+s+t\, = \, 2N}
(-1)^{r+st} \, m_{r+t+1}
\Bigl(\mbf{\CCW}^{\otimes{r}}\otimes m_{s} \bigl( \mbf{\CCW}^{\otimes{s}}\bigr) \otimes \mbf{\CCW}^{\otimes{t}}\Bigr)
\nonumber \\
& = \,
\smtwo{s = 0,2,\dots \text{(even)}}{r+t+1 =  1,3,\dots \text{(odd)}} \!
(-1)^{r} \, m_{r+t+1}
\Bigl(\mbf{\CCW}^{\otimes{r}}\otimes m_{s} \bigl( \mbf{\CCW}^{\otimes{s}}\bigr) \otimes \mbf{\CCW}^{\otimes{t}}\Bigr)
\nonumber \\
& \hspace{13pt} + \!
\smtwo{s = 1,3,\dots \text{(odd)}}{r+t+1 =  2,4,\dots \text{(even)}} \!
\mathop{\underbrace{(-1)^{r+t}}}\limits_{(-1)} \ m_{r+t+1}
\Bigl(\mbf{\CCW}^{\otimes{r}}\otimes m_{s} \bigl( \mbf{\CCW}^{\otimes{s}}\bigr) \otimes \mbf{\CCW}^{\otimes{t}}\Bigr)
\nonumber \\
& = \,
\mathop{\sum}\limits_{k \, = \, 1}^{\infty}
\,
\Biggl(\mathop{\sum}\limits_{r \, = \, 0}^{2k-2} (-1)^r\Biggr)
\,
m_{2k-1}
\left(\mbf{\CCW}^{\otimes{(2k-1)}}\right)
\nonumber \\ & \hspace{13pt}
-
\mathop{\sum}\limits_{k \, = \, 1}^{\infty}
\,
\mathop{\sum}\limits_{r \, = \, 0}^{2k-1}
\,
m_{2k}
\left(\mbf{\CCW}^{\otimes{r}} \otimes \mbf{\CCW'} \otimes \mbf{\CCW}^{\otimes{(2k-r-1)}}\right)
\nonumber \\
& = \,
\mbf{\CCW'}
-
\mathop{\sum}\limits_{k \, = \, 1}^{\infty}
\,
\mathop{\sum}\limits_{r \, = \, 0}^{2k-1}
\,
m_{2k}
\left(\mbf{\CCW}^{\otimes{r}} \otimes \mbf{\CCW'} \otimes \mbf{\CCW}^{\otimes{(2k-r-1)}}\right)
\nonumber
\end{align}
and hence
\be{tmpeq-x2}
\mbf{\CCW'} \, = \,
\mathop{\sum}\limits_{k \, = \, 1}^{\infty}
\,
\mathop{\sum}\limits_{r \, = \, 0}^{2k-1}
\,
m_{2k}
\left(\mbf{\CCW}^{\otimes{r}} \otimes \mbf{\CCW'} \otimes \mbf{\CCW}^{\otimes{(2k-r-1)}}\right) \,.
\ee
The homogeneous part of degree $2n-1$ in Eq.~(\ref{tmpeq-x2}) gives
\be{tmpeq-x3}
\CCW'_n \, = \,
\mathop{\sum}\limits_{k \, = \, 1}^{\infty}
\
\smtwo{n_1+\cdots+n_{2k} = n+k}{n_1,\dots,n_{2k} \geqslant 1}
\
\mathop{\sum}\limits_{r \, = \, 0}^{2k-1} \,
m_{2k}
\left(\CCW_{n_1},\dots,\CCW_{n_r} \CCW'_{n_{r+1}}, \CCW_{n_{r+2}},\dots, \CCW_{n_{2k}}\right) \,
\ee
($\mbf{\CCW'} = \mathop{\sum}\limits_{n \, = \, 1}^{\infty} \CCW'_n$, $\deg \CCW'_n = 2n-1$).
Now Eq.~(\ref{tmpeq-x3}) easily implies by induction on $n$ that $\CCW'_n = 0$ for all $n$,
since $\CCW_1=m_0$ and $\CCW'_1 = m_1 (m_0) = 0$.

\medskip

Next we discuss the independence of our Chern character on the choice of ``connection''.

\begin{proposition}\label{indcon}  Let $A$ be a finite dimensional curved $A_{\infty}$-algebra. For $\gamma \in A_{1}$ let $A^{\gamma}$ denote the deformed algebra described in Proposition \ref{ainfdeform}. Assume that the series (\ref{deforma}) are uniformly convergent with respect to $\gamma$ on compact subsets of $A_{1}$. Suppose further that for every $\gamma \in A_{1}$ we have a Chern-Weil triple $(A^{\gamma}, \Phi, C)$. Then the first two components of the Chern characters of these Chern-Weil triples do not depend on $\gamma$.
\end{proposition}

\begin{proof} We adapt the argument given in \cite[Section 5.2]{Po}. We denote by $d_{R}$ the DeRham differential on $\Omega^{\bullet}(A_{1},A)$, the complex of smooth differential forms on $A_{1}$ with values in $A$. Since $A$ is finite dimensional we can define for every $\gamma \in A_{1}$ the tautological 1-form $d_{R}\gamma \in \Omega^{1}(A_{1},A_{1})$. Then it is not hard to check, differentiating term-by-term with respect to $\gamma$, that we have
\be{twodif}d_{R}m_{0}^{\gamma}= m_{1}^{\gamma} d_{R}\gamma,
\ee
and the claim for the first component of the Chern character easily follows. A simple computation using (\ref{twodif}) and (\ref{curvedAinfinity}) shows that
$d_{R}\CCW_{2}^{\gamma}=m_{1}^{\gamma} \omega$ for some $\omega \in \Omega^{1}(A_{1},A)$ which completes the proof.
\end{proof}

\subsection{Morphisms of Chern-Weil triples}

\begin{definition}\label{cwmorph} Let $(A,\Phi,C)$ and $(B,\Psi,D)$ be two Chern-Weil triples.
Let $F:A \rightarrow B$ be a morphism of $A_{\infty}$-algebras and let $f:C \rightarrow D$ be a chain map.
We say that the pair $(F,f)$ is a {\bf morphism of Chern-Weil triples} if for every
even $x$ in $A$, every odd $n>1$
and all even $x_{1},\ldots,x_{n} \in A$ the following conditions hold.
\begin{eqnarray*}\text{(i)} & \Psi (F_{1}(x))=f(\Phi(x))\\
\text{(ii)} & \sum_{\sigma \in \textbf{S}_{n}}\Psi(F_{n}(x_{\sigma (1)},\ldots, x_{\sigma(n)}))=0.
\end{eqnarray*}
\end{definition}

Clearly Chern-Weil triples with this notion of morphism do {\em not} form a category.
It turns out however that the Chern character we have defined is natural with respect to certain class of morphisms.
This will easily follow from our next lemma.

Let us introduce as in Sect.~\ref{PreCh} the pre-Chern elements
 $\CCW_n^A$ and $\CCW_n^B$, and their formal sums $\mbf{\CCW_A}$ and $\mbf{\CCW_B}$.

\begin{lemma}\label{LM-X}
Let $F = \{F_k\} : \bigl(A,\{m_k^A\}\bigr) \to \bigl(B,\{m_k^B\}\bigr)$ be a morphism
of two curved $A_{\infty}$-algebras such that  $F_k$ vanishes for all even $k$
Then
\begin{eqnarray}\label{evenF}
\mbf{\CCW_B} &&  = \, \mathop{\sum}\limits_{k \, = \, 1}^{\infty}
F_{2k-1} \Bigl(\mbf{\CCW}^{\otimes (2k-1)}_{\mbf{A}}\Bigr) \,,
\end{eqnarray}
or equivalently
\begin{eqnarray}
\label{evenF-1}
\CCW^B_n &&  = \,
\mathop{\sum}\limits_{k \, = \, 1}^{\infty}
\
\smtwo{n_1+\cdots+n_{2k-1}}{= n+k-1}
\,
F_{2k-1}
\left(\CCW_{n_1}^A,\dots,\CCW_{n_{2k-1}}^A\right) \,.
\end{eqnarray}
(Note that the sum over $k$ in (\ref{evenF-1}) is in fact finite.)
\end{lemma}

\begin{proof}
Consider first a general morphism $F = \{F_k\} : \bigl(A,\{m_k^A\}\bigr) \to \bigl(B,\{m_k^B\}\bigr)$
of two curved $A_{\infty}$-algebras. We can transfer $\mbf{\CCW_A}$ to the following elements in $B$
\be{CCw}
\mbf{\CCw^{\pm}_{A,B}} \, := \, \mathop{\sum}\limits_{k \, = \, 1}^{\infty}
(\pm 1)^{k-1} F_k \Bigl(\mbf{\CCW}^{\otimes k}_{\mbf{A}}\Bigr) \,.
\ee
Thus $\mbf{\CCw^{\pm}_{A,B}}$ are the formal generating series for $\CCw_{n}^{A,B}$ which are defined by
\[
\mbf{\CCw^{\pm}_{A,B}} \, := \, \mathop{\sum}\limits_{n \, = \, 2}^{\infty} (\pm 1)^n \, \CCw_{n}^{A,B}
\,, \qquad
\deg \, \CCw_{n}^{A,B} \, = \, n  \,
\]
and Eq.~(\ref{CCw}) reads
\be{CCw1}
\CCw_{n}^{A,B} \, = \,
\mathop{\sum}\limits_{k \, = \, 1}^{\infty}
\
\smtwo{2n_1+\cdots+2n_k}{= n+k-1}
\,
F_k
\left(\CCW_{n_1}^A,\dots,\CCW_{n_k}^A\right) \,.
\ee
It is convenient to split $\mbf{\CCw^{+}_{A,B}}$ into even and odd part, $\mbf{\CCw^0_{A,B}}$ and $\mbf{\CCw^1_{A,B}}$,
respectively, so that
\[
\mbf{\CCw^{\pm}_{A,B}} \, = \, \mbf{\CCw^0_{A,B}} \pm \mbf{\CCw^1_{A,B}}
\,,\qquad
\mbf{\CCw^{0/1}_{A,B}} \, = \, \frac{1}{2} \bigl( \mbf{\CCw^{+}_{A,B}} \pm \mbf{\CCw^{-}_{A,B}} \bigr)\,,
\]
\[
\mbf{\CCw^0_{A,B}} \, := \, \mathop{\sum}\limits_{k \, = \, 1}^{\infty}
F_{2k-1} \Bigl(\mbf{\CCW}^{\otimes (2k-1)}_{\mbf{A}}\Bigr)
\,,\qquad
\mbf{\CCw^1_{A,B}} \, := \, \mathop{\sum}\limits_{k \, = \, 1}^{\infty}
F_{2k} \Bigl(\mbf{\CCW}^{\otimes 2k}_{\mbf{A}}\Bigr) \,.
\]

We now proceed as in Sect.~\ref{PreCh} and start with the identity for $A_{\infty}$-morphism (\ref{ainmorph}):
\[
\sum_{r+s+t \, = \, n}\!(-1)^{r+st} \, F_{r+t+1}
\Bigl(\mbf{\CCW}^{\otimes{r}}_{\mbf{A}}\otimes m^A_{s} \bigl(\mbf{\CCW}^{\otimes s}_{\mbf{A}}\bigr) \otimes
\mbf{\CCW}^{\otimes{t}}_{\mbf{A}}\Bigr)
\]
\be{ainmorph-1}
\, = \,
\hspace{-7pt}
\smtwo{1 \, \leqslant \, q \, \leqslant \, n}{i_{1}+ \ldots + i_{q} \, = \, n}
\hspace{-9pt}
(-1)^{w} \,
m^B_{q} \Bigl(F_{i_{1}} \bigl(\mbf{\CCW}^{\otimes i_1}_{\mbf{A}}\bigr),\dots,
F_{i_{q}} \bigl(\mbf{\CCW}^{\otimes i_q}_{\mbf{A}}\bigr)\Bigr)\,.
\ee
Using Eq. (\ref{ainfhom1}) and summing (\ref{ainmorph-1}) over $n$, we obtain
\be{ainmorph-2}
\sum_{r,s,t \, \geqslant \, 0}\!(-1)^{r+st} \, F_{r+t+1}
\Bigl(\mbf{\CCW}^{\otimes{r}}_{\mbf{A}}\otimes m^A_{s} \bigl(\mbf{\CCW}^{\otimes s}_{\mbf{A}}\bigr) \otimes
\mbf{\CCW}^{\otimes{t}}_{\mbf{A}}\Bigr)
\ee
\[
\, = \,
\hspace{-7pt}
\smtwo{q \, \geqslant \, 0 }{i_{1},\ldots,i_{q} \, \geqslant \, 1}
\hspace{-9pt}
m^B_{q}
\Bigl(
(\pm 1)^{i_1-1} \, F_{i_{1}} \bigl(\mbf{\CCW}^{\otimes i_1}_{\mbf{A}}\bigr),\dots,
(-1)^{i_{q-1}-1} \, F_{i_{q-1}} \bigl(\mbf{\CCW}^{\otimes i_{q-1}}_{\mbf{A}}\bigr),
F_{i_{q}} \bigl(\mbf{\CCW}^{\otimes i_q}_{\mbf{A}}\bigr)
\Bigr)\,
\]
The right hand side of Eq.~(\ref{ainmorph-2}) contains alternating sign factors $(-1)^{(q-\ell)(i_{\ell}-1)}$
and simply gives
\be{PME}
\mathop{\sum}\limits_{q \, = \, 0}^{\infty}
m_q^B \bigl(\mbf{\CCw^{\pm}_{A,B}},\dots,\mbf{\CCw^{-}_{A,B}},\mbf{\CCw^{+}_{A,B}}\bigr) \, ,
\ee
where the signs in the arguments alternate, so that sign in the first argument is $+$  when $q$ is odd and $-$ when $q$ is even.
The  left hand side of Eq.~(\ref{ainmorph-2}) can be transformed as follows
\begin{align}\label{calc3}
&
\sum_{r,s,t \, \geqslant \, 0}\!(-1)^{r+st} \, F_{r+t+1}
\Bigl(\mbf{\CCW}^{\otimes{r}}_{\mbf{A}}\otimes m^A_{s} \bigl(\mbf{\CCW}^{\otimes s}_{\mbf{A}}\bigr) \otimes
\mbf{\CCW}^{\otimes{t}}_{\mbf{A}}\Bigr)
\nonumber \\ &
= \,
\sum_{r,t \, \geqslant \, 0} \, (-1)^{r} \, F_{r+t+1}
\Biggl(\mbf{\CCW}^{\otimes{r}}_{\mbf{A}}\otimes
\Biggl(\sum_{s \, = \, 0,2,\dots \text{ (even)}}
m^A_{s} \bigl(\mbf{\CCW}^{\otimes s}_{\mbf{A}}\bigr) \Biggr) \otimes \mbf{\CCW}^{\otimes{t}}_{\mbf{A}}\Biggr)
\nonumber \\ &
\hspace{13pt} + \,
\sum_{r,t \, \geqslant \, 0} \,
(-1)^{r+t} \ F_{r+t+1}
\Biggl(\mbf{\CCW}^{\otimes{r}}_{\mbf{A}} \otimes
\Biggl(\sum_{s \, = \, 1,3,\dots \text{ (odd)}}
m^A_{s} \bigl(\mbf{\CCW}^{\otimes s}_{\mbf{A}}\bigr) \Biggr) \otimes \mbf{\CCW}^{\otimes{t}}_{\mbf{A}}\Biggr)
\nonumber \\ &
= \,
\sum_{r,t \, \geqslant \, 0} \, (-1)^{r} \, F_{r+t+1}
\Bigl(\mbf{\CCW}^{\otimes{r}}_{\mbf{A}}\otimes \mbf{\CCW_A} \otimes \mbf{\CCW}^{\otimes{t}}_{\mbf{A}}\Bigr)
\nonumber \\ &
= \,
\sum_{\ell = 1}^{\infty} \Biggl( \sum_{r \, = \, 0}^{\ell-1} \, (-1)^{r} \Biggr) \, F_{\ell}
\Bigl(\mbf{\CCW}^{\otimes{\ell}}_{\mbf{A}}\Bigr)
\, = \,
\sum_{k = 1}^{\infty} \, F_{2k-1}
\Bigl(\mbf{\CCW}^{\otimes{(2k-1)}}_{\mbf{A}}\Bigr) \, = \, \mbf{\CCw^0_{A,B}} \,.
\nonumber
\end{align}
Thus Eq.~(\ref{ainmorph-2}) gives us
\begin{eqnarray}\label{ainmorph-3}
\hspace{-50pt} &&
\mbf{\CCw^0_{A,B}} \, = \,
\mathop{\sum}\limits_{q \, = \, 0}^{\infty}
m_q^B \bigl(\mbf{\CCw^{\pm}_{A,B}},\dots,\mbf{\CCw^{-}_{A,B}},\mbf{\CCw^{+}_{A,B}}\bigr) \,,
\\ \label{ainmorph-3-1}
\hspace{-50pt} &&
\Longleftrightarrow \hspace{6pt}
\left\{\raisebox{46pt}{\hspace{-4pt}}\right.
\begin{array}{rl}
\CCw_{2n}^{A,B} \hspace{-6pt} & = \,
\mathop{\displaystyle \sum}\limits_{q \, = \, 0}^{\infty}
\
{\displaystyle \smthr{n_1+\cdots+n_q}{= 2n+q-2}{n_1,\dots,n_q \geqslant 2}}
\,
(-1)^{\mathop{\sum}\limits_{\ell \, = \, 1}^q (q-\ell) \, n_{\ell}} \,
m_q^B \bigl(\CCw_{n_1}^{A,B},\dots,\CCw_{n_q}^{A,B}\bigr)
\\
0 \hspace{-1pt} & = \,
\mathop{\displaystyle \sum}\limits_{q \, = \, 0}^{\infty}
\
{\displaystyle \smthr{n_1+\cdots+n_q}{= 2n+q-1}{n_1,\dots,n_q \geqslant 2}}
\,
(-1)^{\mathop{\sum}\limits_{\ell \, = \, 1}^q (q-\ell) \, n_{\ell}} \,
m_q^B \bigl(\CCw_{n_1}^{A,B},\dots,\CCw_{n_q}^{A,B}\bigr)
\,.
\end{array}
\end{eqnarray}


Now let us assume that the $A_{\infty}$-morphism $F = \{F_k\}$ is such that $F_{k}=0$  for all even $k$.
It follows that
\[
\mbf{\CCw^{+}_{A,B}} \, = \, \mbf{\CCw^{-}_{A,B}} \, = \, \mbf{\CCw^0_{A,B}}
\,,\qquad \mbf{\CCw^1_{A,B}} \, = \, 0 \,.
\]
Hence Eq.~(\ref{ainmorph-3}) reads
\[
\mbf{\CCw^0_{A,B}} \, = \,
\mathop{\sum}\limits_{q \, = \, 0}^{\infty}
m_q^B \Bigl(\bigl(\mbf{\CCw^0_{A,B}}\bigr)^{\otimes{q}}\Bigr) \,,
\]
which is equivalent to
\[
\mbf{\CCw^0_{A,B}} \, = \,
\mathop{\sum}\limits_{k \, = \, 0}^{\infty}
m_{2k}^B \Bigl(\bigl(\mbf{\CCw^0_{A,B}}\bigr)^{\otimes{2k}}\Bigr)
\,, \qquad
\mathop{\sum}\limits_{k \, = \, 1}^{\infty}
m_{2k-1}^B \Bigl(\bigl(\mbf{\CCw^0_{A,B}}\bigr)^{\otimes{(2k-1)}}\Bigr) \, = \, 0 \,
\]
(since $\mbf{\CCw^0_{A,B}}$ is even).
But these equations recursively determine $\mbf{\CCW_B}$ and  we conclude that
\[
\mbf{\CCW_B} \, = \, \mbf{\CCw^0_{A,B}} \,, \qquad \text{i.e.,} \qquad
\CCW_k^B \, = \, \CCw_{2k}^{A,B}  \quad k=1,2,\dots
\]
The last identity together with Eq.~(\ref{CCw}) implies Eq. (\ref{evenF}) and the lemma is proved.
%
\end{proof}

\begin{proposition}\label{naturality}
Let $(F,f)$ be a morphism of Chern-Weil triples from  $(A,\Phi,C)$ to $(B,\Psi,D)$
such that $F_{k}$ vanishes for all even $k$. Then
\[f(\mathcal{C}^{A}_{n})=\mathcal{C}^{B}_{n}.\]
\end{proposition}

\begin{proof}
In view of (\ref{ainfhom1})
it suffices to show that for every $n$  the expression
\[
\sum_{\TRE \, \in \, \CET(2n)}F_{1}(\mathcal{M}_{\TRE}^{A})-{\mathcal{M}_{\TRE}^{B}}
\, = \,
F_1 \bigl(\CCW_n^A\bigr)  - \CCW_n^B
\]
is in the kernel of $\Psi$.
But Eq.~(\ref{evenF-1}) gives us
\[
\CCW^B_n \, - \, F_1 \bigl(\CCW_n^A\bigr) \, = \,
\mathop{\sum}\limits_{k \, = \, 2}^{\infty}
\
\smtwo{n_1+\cdots+n_{2k-1}}{= n+k-1}
\,
F_{2k-1}
\left(\CCW_{n_1}^A,\dots,\CCW_{n_{2k-1}}^A\right) \,,
\]
and we conclude that the right hand side is indeed in the kernel of $\Psi$ according to condition (ii) in Definition~(\ref{cwmorph}).
\end{proof}

In what follows we will apply the above proposition only in the (trivial) case when $F$ is a strict morphism
of curved $A_{\infty}$-algebras, i.e. when $F_{k}$ vanishes for every $k>1$.


\subsection{Examples}
Let $M$ be a smooth manifold and let $\Omega^{\bullet}(M,\text{End}(E))$ be the curved dg algebra from Example \ref{ex1section1}. Then obviously $(\Omega^{\bullet}(M,\text{End}(E)),
\text{Tr}, \Omega^{\bullet}(M))$ (where $\text{Tr}$ denotes the matrix trace) is a Chern-Weil triple. The Chern character of this triple coincides with the Chern character of $E$.

Let $E$ be a vector bundle over a finite polyhedron $X$ and  let $\Omega^{\bullet}(X,\textbf{M}_{l}(\C))$ be the curved dg algebra from Example \ref{transaex} obtained from a
 piece-wise smooth idempotent $P$ representing $E$. We shall
now explain how one can extend the classical Chern-Weil theory to the piece-wise smooth setting. Indeed, one can apply the general construction of
\cite[Chapter 1]{Ka} to the dg algebra $\Omega^{\bullet}(X)$ to obtain a Chern character
$$\text{Ch}: K_{0}(\Omega^{0}(X)) \rightarrow H^{\bullet}_{dR}(X)$$
whose $k$-th component $\text{Ch}_{k}(P)$ is given by the usual formula
$\text(1/k!){Tr}(P(dP)^{k})$. Then one can define (see \cite{Ka}) Chern classes $c_{k}(P)$ that are related to the Chern character by the formula
$$\text{Ch}_{k}(P)=\frac{1}{k!}Q_{k}(c_{1}(P),\ldots,c_{k}(P)),$$
where $Q_{k}$ denotes the $k$-th elementary symmetric function.

Clearly $\text{Ch}$ is natural (with respect to piece-wise smooth maps) and coincides with the Chern character given by Chern-Weil theory when $P$ represents a smooth bundle over a smooth compact manifold. It follows that $c_{k}(P)$ satisfy the axioms for Chern classes stated in \cite[Chapter 1,\S 4]{Hi} and therefore coincide with the standard Chern classes defined for bundles over arbitrary compact space $X$. In other words, the
Chern character of  the Chern-Weil triple $(\Omega^{\bullet}(X,\textbf{M}_{l}(\C)),
\text{Tr}, \Omega^{\bullet}(X))$ is the Chern character of $E$.

It is not hard to see that  $(C^{\bullet}(X,\textbf{M}_{l}(\C)),\text{Tr},C^{\bullet}(X))$,  where
$C^{\bullet}(X,\textbf{M}_{l}(\C))$ is the curved $A_{\infty}$-algebra from Example \ref{transaex}, is a Chern-Weil triple. Indeed, condition (i) in Definition \ref{cwtriple} is easily checked using Proposition \ref{expform} and the fact that the trace vanishes on commutators. Condition (ii) directly follows from Lemma \ref{cyclsym} since we can write the sum over all permutations as a sum of subsums over cyclic permutations.

Similarly, the curved $A_{\infty}$-morphism from $C^{\bullet}(X,\textbf{M}_{l}(\C))$ to $\Omega^{\bullet}(X,\textbf{M}_{l}(\C))$ given by Proposition \ref{transa} induces a moprhism of Chern-Weil triples. This morphism however does not satisfy
the assumptions of Proposition \ref{naturality}.  Thus, we choose a different route to show that the Chern character of the Chern-Weil triple defined by the curved $A_{\infty}$-algebra $C^{\bullet}(X,\textbf{M}_{l}(\C))$ coincide with the Chern character of the bundle $E$.

\begin{lemma}\label{strict} Let $X$ be locally finite polyhedron and let $C^{\bullet}(X)$ be the $A_{\infty}$-algebra defined after Theorem \ref{contraction}. There exists a strict quasi-isomorphism of $A_{\infty}$-algebras from $C^{\bullet}(X)$ to $C^{\bullet}(\widetilde{X})$ for every sufficiently fine subdivision $\widetilde{X}$ of $X$.
\end{lemma}

\begin{proof} If  $\widetilde{X}$ is a fine enough subdivision of $X$ then there exists a simplicial map
$\varphi : \widetilde{X} \rightarrow X $ which can be obtained by applying the simplicial approximation
theorem to the identity map on $X$. The map $\varphi$ clearly induces a dg algebra quasi-isomorphism \[\varphi ^{*}: \Omega^{\bullet}(X)
\rightarrow \Omega^{\bullet}(\widetilde{X}).\]
It follows directly from the definition of the homotopy $H$ given in \cite[Chapter 2]{Du} that it commutes with simplicial maps. In particular, we have
\[\varphi ^{*}  H_{X}= H_{\widetilde{X}}  \varphi ^{*}\]

Let $\widetilde{W}_{X}:C^{\bullet}(X) \rightarrow \Omega ^{\bullet}(X)$ and $\widetilde{R}_{\widetilde{X}}:
\Omega ^{\bullet}(\widetilde{X}) \rightarrow C^{\bullet}(\widetilde{X})$
 be the $A_{\infty}$-morphisms given by Proposition \ref{transa}. Then Lemma \ref{strictmor} implies that the composition
 $ \widetilde{R}_{\widetilde{X}}  \varphi ^{*}  \widetilde{W}_{X}$
   is a strict morphism of $A_{\infty}$-algebras.
\end{proof}

\begin{theorem} The images under the De Rham isomorphism of the Chern character of the Chern-Weil triple $(C^{\bullet}(X,\textbf{M}_{l}(\C)),\text{Tr},C^{\bullet}(X))$ coincide with the Chern character of $E$.
\end{theorem}

\begin{proof} Given a polyhedron $Y$ and $\gamma \in \Omega ^{1}(Y,\textbf{M}_{l}(\C))$, denote the curved dg algebra
obtained by deforming the  dg algebra $\Omega ^{\bullet}(Y,\textbf{M}_{l}(\C))$ along $\gamma$ by $\Omega ^{\bullet}(Y)^{\gamma}$.
Proceeding as in the proof of Lemma \ref{strict}, we see that  there exists a sequence of subdivisions $\{X_{j}\}_{j=1}^{\infty}$ of $X$ satisfying the following conditions:

(1) For every $j$ there exists a simplicial map $\varphi _{j}:X_{j} \rightarrow X$ that induces a morphism of curved dg algebras

\[\varphi_{j}^{*}: \Omega^{\bullet}(X)^{\gamma}
\rightarrow \Omega^{\bullet}(X_{j})^{\varphi_{j}^{*}(\gamma)}\]

and isomorphisms \[\varphi_{\Omega,j}^{*}:H ^{\bullet}_{dR}(X) \rightarrow H ^{\bullet}_{dR}(X_{j}),\]
\[\varphi_{C,j}^{*}:H ^{\bullet}_{simpl}(X) \rightarrow  H ^{\bullet}_{simpl}(X_{j})\] on De Rham and simplicial cohomology respectively, that commute with the De Rham isomorphisms.

(2) The fullness of $X_{j}$ is bounded from below and the mesh of $X_{j}$ tends to 0 as $j\rightarrow \infty$ (see \cite{Do} for the relevant definitions).

\medskip

Let $C^{\bullet}(X)^{\gamma}$ (respectively $C^{\bullet}(X_{j})^{\varphi_{j}^{*}(\gamma)}$) denote the curved $A_{\infty}$-algebras obtained by transferring the curved dg structure on  $\Omega^{\bullet}(X)^{\gamma}$ (respectively $\Omega^{\bullet}(X_{j})^{\varphi_{j}^{*}(\gamma)}$) as in Example \ref{transaex}. Then Lemma \ref{strictmor} implies that there exists a strict morphism of curved $A_{\infty}$-algebras
between  $C^{\bullet}(X)^{\gamma}$ and $C^{\bullet}(X_{j})^{\varphi_{j}^{*}(\gamma)}$
which clearly coincides with $\varphi_{C,j}^{*}$ at the cohomology level.

Let us denote the Chern characters of the Chern-Weil triples defined by the curved dg algebras $\Omega^{\bullet}(X)^{\gamma}$ and
$\Omega^{\bullet}(X_{j})^{\varphi_{j}^{*}(\gamma)}$ by $\textbf{Ch}^{\bullet}_{\Omega}(X)$ and $\textbf{Ch}^{\bullet}_{\Omega}(X_{j})$ respectively, and the Chern characters of the Chern-Weil triples defined by the curved $A_{\infty}$-algebras $C^{\bullet}(X)^{\gamma}$ and
$C^{\bullet}(X_{j})^{\varphi_{j}^{*}(\gamma)}$ by $\textbf{Ch}^{\bullet}_{C}(X)$ and $\textbf{Ch}^{\bullet}_{C}(X_{j})$ respectively. Then by naturality of the Chern character under strict $A_{\infty}$-morphisms we have
\be{allsame} \overline{W}_{X}((\varphi_{C,j}^{*})^{-1}(\textbf{Ch}^{\bullet}_{C}(X_{j})))=\overline{W}_{X}(\textbf{Ch}^{\bullet}_{C}(X))
\ee
for every $j$, where  $\overline{W}_{X}$  denotes the De Rham isomophism induced by the Whitney map $W_{X}$.

Now using condition (2) above one can show exactly as in \cite[Theorems 5.4 and 5.12]{W} that for all $p>2$ and all forms $\omega _{1}, \ldots, \omega _{p}$ in $\Omega^{\bullet}(X, \textbf{M}_{l}(\C))$ one has
\[\lim _{j \rightarrow \infty}\|W_{X_{j}}(m_{2}^{j}(R_{X_{j}}(r_{j}(\omega _{1})),R_{X_{j}}(r_{j}(\omega _{2})))-\omega _{1} \wedge \omega _{2}\|_{j}=0,\]
\[\lim _{j \rightarrow \infty}\|W_{X_{j}}(m_{p}^{j}(R_{X_{j}}(r_{j}(\omega _{1})),\ldots , R_{X_{j}}(r_{j}(\omega _{p})))\|_{j}=0,\]
where $ \{m_{p}^{j}\}$ is the curved $A_{\infty}$-structure on $C^{\bullet}(X_{j})^{\varphi_{j}^{*}(\gamma)}$, the map \[r_{j}:\Omega^{\bullet}(X, \textbf{M}_{l}(\C)) \rightarrow \Omega^{\bullet}(X_{j}, \textbf{M}_{l}(\C))\] is the
natural restriction of forms, and $\|.\|_{j}$ is the $L_{2}$-norm on $\Omega^{\bullet}(X_{j}, \textbf{M}_{l}(\C))$. Using this, it is not hard to show that
\be{interm}\lim _{j \rightarrow \infty} \|\overline{W}_{X_{j}}(\textbf{Ch}^{\bullet}_{C}(X_{j}))-\textbf{Ch}^{\bullet}_{\Omega}(X_{j})\|_{j}=0.
\ee
One can easily verify that the maps $\varphi_{j}^{*}$ are uniformly bounded from below. It follows that $(\varphi_{\Omega,j}^{*})^{-1}$ are uniformly bounded from above, hence Eq. (\ref{interm}) implies that
\be{intermphi}\lim _{j \rightarrow \infty} \|(\varphi_{\Omega,j}^{*})^{-1}(\overline{W}_{X_{j}}(\textbf{Ch}^{\bullet}_{C}(X_{j}))
-\textbf{Ch}^{\bullet}_{\Omega}(X)\|_{j}=0.
\ee

Now since
\[(\varphi_{\Omega,j}^{*})^{-1}  \overline{W}_{X_{j}}=\overline{W}_{X}  (\varphi_{C,j}^{*})^{-1},\]
comparing Eqs. (\ref{allsame}) and (\ref{intermphi}) we conclude that $\overline{W}_{X}(\textbf{Ch}^{\bullet}_{C}(X))$ coincides with $\textbf{Ch}^{\bullet}_{\Omega}(X)$, the Chern character of $E$.
\end{proof}

Finally we note that Lemma \ref{secondcw} imples that the curved  $A_{\infty}$-algebra $C^{\bullet}(X,\textbf{M}_{l}(\textbf{k}))^{\gamma}$ from Proposition \ref{exadef} defines a Chern-Weil triple $(C^{\bullet}(X,\textbf{M}_{l}(\textbf{k}))^{\gamma},\text{Tr}, C^{\bullet}(X)) $.
However, since $C^{\bullet}(X,\textbf{M}_{l}(\textbf{k}))^{\gamma}$ is a deformation of an ordinary $A_{\infty}$-algebra, Proposition \ref{indcon} shows that one cannot obtain notrivial Chern character from this Chern-Weil triple.

\medskip

\noindent
{\bf Acknowledgment.}
N.N. acknowledges partial support from the program ``Quantum field theory on curved space-times and curved target-spaces'' of the Erwin Schrodinger Instutute for Mathematical Physics.

\end{document}